\documentclass{amsart}
\usepackage{amsmath,amsthm,amsfonts,amssymb,amscd}
\usepackage[all]{xy}
\usepackage[english]{babel}
\usepackage{graphicx,color}
\usepackage[T1]{fontenc}
\usepackage[latin1]{inputenc}
\usepackage{ae}

\usepackage[mathscr]{eucal}
\usepackage{amsmath,amssymb}
\usepackage{graphics}
\usepackage{multirow}
\usepackage{hyperref}
\usepackage[active]{srcltx}

\usepackage{vmargin}

\setpapersize{A4}
\setmargins{3.0cm}       % margen izquierdo
{1.5cm}                        % margen superior
{13.0cm}                      % anchura del texto
{23.42cm}                    % altura del texto
{10pt}                           % altura de los encabezados
{1cm}                           % espacio entre el texto y los encabezados
{0pt}                             % altura del pie de página
{2cm}

\newtheorem{thm}{Theorem}[section]

\newtheorem{cor}[thm]{Corollary}
\newtheorem{prop}[thm]{Proposition}

\newtheorem{lemma}[thm]{Lemma}

\theoremstyle{definition}

\newtheorem{remark}[thm]{Remark}
\newtheorem{example}[thm]{Example}
\newtheorem{question}[thm]{Question}

\newcommand{\nc}{\newcommand}
\nc {\hh}{\check{h}}
\nc {\DD}{\mathcal{D}}
\nc {\RR}{\mathcal{R}}
\nc {\Pp}{\mathbb{P}}
\nc {\Ss}{\mathcal{S}}
\nc {\PP}{\mathbb{P}^{2}}
\nc {\Pd}{ \check{\mathbb{P}}^{2}}
\nc {\WW}{\mathcal{W}}
\nc {\Sym}{\mathrm{Sym}}
\nc {\OO}{\mathcal{O}}
\nc {\CC}{\mathbb{C}}
\nc {\EE}{\mathcal{E}}
\nc {\MM}{\mathcal{M}}
\nc {\KK}{\mathcal{K}}
\nc {\PW}{\mathcal{P}}
\nc {\NW}{\mathcal{N}_{\WW}}
\nc {\FF}{\mathcal{F}}
\nc {\GG}{\mathcal{G}}
\nc {\ZZ}{\mathcal{Z}}
\nc {\LL}{\mathcal{L}}
\nc {\HH}{\mathcal{H}}
\nc {\NN}{\mathcal{N}}
\nc {\VV}{\mathcal{V}}
\nc {\Ww}{\mathbb{W}}
\nc {\QQ}{\mathbb{Q}}
\nc {\II}{\mathcal{I}}
\nc {\tang}{\mathrm{Tang}}
\nc {\Disc}{\mathbb{D}}

\date{\today}

\begin{document}

\title[Projective structures, neighborhoods of curves and Painlev\'e equations]{Projective structures, neighborhoods of rational curves and Painlev\'e equations}
\author[M. Falla Luza, F. Loray]
{Maycol Falla Luza$^{1}$, Frank Loray$^2$}
\address{\newline $1$ UFF, Universidad Federal Fluminense, rua M\'ario Santos Braga S/N, 
Niter\'oi, RJ, Brazil, Brasil\hfill\break
$2$ Univ Rennes 1, CNRS, IRMAR, UMR 6625, F-35000 Rennes, France} 
\email{$^1$ maycolfl@gmail.com} \email{$^2$ frank.loray@univ-rennes1.fr}
\thanks{The first author is supported by CNPq. 
The second author is supported by CNRS, Henri Lebesgue Center and ANR-16-CE40-0008 project ``{\it Foliage}''. The authors also thank Brazilian-French Network in Mathematics and CAPES-COFECUB Project Ma 932/19
``{\it Feuilletages holomorphes et int\'eractions avec la g\'eom\'etrie}''.
The authors thank Jorge Vit\'orio Pereira and Luc Pirio for numerous discussions, and Sergei Ivachkovitch 
for letting us know about Mishustin's work.}

\begin{abstract}
We investigate the duality between local (complex analytic) projective structures on surfaces
and two dimensional (complex analytic) neighborhoods of rational curves having self-intersection $+1$. 
We study the analytic classification, existence of normal forms, pencil/fibration decomposition, 
infinitesimal symmetries. We deduce some transcendental result about Painlev\'e equations.
Part of the results were announced in \cite{CRAS}; an extended version is available in \cite{MaycolFrank}.
\end{abstract}
\maketitle
\setcounter{tocdepth}{1}
\tableofcontents

\section{Introduction}

Duality between lines and points in $\mathbb P^2$ has a nice non linear generalization
which goes back at least to the works of \'Elie Cartan, and even appears in Alfred Koppisch's thesis in 1905.
The simplest (or more familiar) setting 
where this duality takes place is when considering the geodesics of a given
Riemannian metric (say, real analytic, or holomorphic) on the neighborhood $U$ 
of the origin in the plane.
The space of geodesics is itself a surface $U^\vee$ that can be constructed as follows.
The projectivized tangent bundle $\mathbb P(T_U)$ is naturally a contact manifold:
given coordinates $(x,y)$ on $U$, the open set of ``non vertical'' directions is parametrized
by triples $(x,y,z)\in U\times\mathbb C$ where $z$ represents the class of the vector field 
$\partial_x+z\partial_y$; the contact structure is therefore given by $dy-zdx=0$.
Each geodesic on $U$ lifts uniquely as a Legendrian curve on $\mathbb P(T_U)$, 
forming a foliation $\mathcal G$ on $\mathbb P(T_U)$. A second Legendrian foliation $\mathcal F$
is defined by fibers of the canonical projection $\pi:\mathbb P(T_U)\to U$.
The two foliations $\mathcal F$ and $\mathcal G$ are transversal, spanning the contact distribution.
Duality arises when we permute of the role of these two foliations.
The space of $\mathcal F$-leaves is the open set $U$; if $U$ is small enough, 
then the space of $\mathcal G$-leaves is also a surface $U^\vee$. 
However, $U^\vee$ is ``semi-global'' in the sense that it contains (projections of) 
$\mathbb P^1$-fibers of $\pi$. If $U$ is a small ball, then it is convex, and we deduce
that any two $\mathbb P^1$-fibers are connected by a unique geodesic ($\mathcal G$-leaf) on $\mathbb P(T_U)$,
i.e. intersect once on $U^\vee$. Finally, we get a $2$-dimensional family (parametrized by $U$)
of rational curves on $U^\vee$ with normal bundle $\mathcal O_{\mathbb P^1}(1)$.
Note that $\PP(T_U)\subset \Pp(T_{U^\vee})$ as contact $3$-manifolds.

In fact, we do not need to have a metric for the construction, but only a collection of curves
on $U$ having the property that there is exactly one such curve passing through a given point
with a given direction. This is what Cartan calls a 
{\it projective structure}, not to be confused with the homonym notion of manifolds locally modelled on $\mathbb P^n$ (see \cite{Dumas,KO}). In coordinates $(x,y)\in U$, such a family
of curves is defined as the graph-solutions to a given differential equation of the form
\begin{equation}\label{eq:ProjStructDiffEq}
y''=A(x,y)+B(x,y)(y')+C(x,y)(y')^2+D(x,y)(y')^3
\end{equation}
with $A,B,C,D$ holomorphic on $U$. Then the geodesic foliation $\mathcal G$
is defined by the trajectories of the vector field 
$$\partial_x+z\partial_y+\left[A+Bz+Cz^2+Dz^3\right]\partial_z.$$
Since it is second order, we know by Cauchy-Kowalevski Theorem
that there is a unique solution curve passing through each point and any non vertical 
direction. That the second-hand is cubic is exactly what we need to insure
the existence and unicity for vertical directions. In a more intrinsec way, 
we can define a projective structure by an affine connection, i.e. 
a (linear) connection $\nabla:T_U\to T_U\otimes\Omega^1_U$ on the tangent bundle.
Then, $\nabla$-geodesics are parametrized curves $\gamma(t)$ on $U$ such that, 
after lifting to $T_U$ as $(\gamma(t),\dot\gamma(t))$, they are in the kernel of $\nabla$.
All projective structures come from an affine connection, but there are many affine connections 
giving rise to the same projective structure: the collection of curves is the same, but with different parametrizations. An example is the Levi-Civita connection associated to a Riemannian metric
and this is the way to see the Riemannian case as a special case of projective structure.
We note that a general projective connection does not come from a Riemannian metric,
see \cite{BDE}.

A nice fact is that the duality construction can be reversed. Given a rational curve $\mathbb P^1\simeq C_0\subset S$ in a surface, having normal bundle $\mathcal O_{C_0}(1)$, 
then Kodaira Deformation Theory tells us that the curve $C_0$ can be locally deformed as a smooth $2$-parameter family
$C_\epsilon$ of curves, likely as for a line in $\mathbb P^2$. We can lift this family as a Legendrian foliation $\mathcal F$
defined on some tube $V\subset\mathbb P(T_{U^\vee})$ and take the quotient: we get a map $\pi:V\to U$
onto the parameter space of the family. Then fibers of $\pi^\vee:\mathbb P(T_{U^\vee})\to U^\vee$
project to the collection of geodesics for a projective structure on $U$.
We thus get a {\it one-to-one correspondance between projectives structures at $(\mathbb C^2,0)$ up
to local holomorphic diffeomorphisms and germs of $(+1)$-neighborhoods $(U^\vee,C_0)$ of $C_0\simeq\mathbb P^1$
up to holomorphic isomorphisms} (see Le Brun's thesis \cite{LeBrun} for many details).

Section \ref{Sec:duality} recalls in more details this duality picture following Arnold's book \cite{Arnold}, 
Le Brun's thesis \cite{LeBrun} and Hitchin's paper \cite{Hitchin}. In particular, the euclidean (or {\bf trivial})
structure by lines, defined by the second order differential equation $y''=0$, corresponds to 
the {\bf linear} neighborhood of the zero section $C_0$ in the total space of $\mathcal O_{C_0}(1)$,
or equivalently of a line in $\mathbb P^2$. But as we shall see, the moduli space of projective structures
up to local isomorphisms has infinite dimension.

We recall in section \ref{ssec:criteriaLin} some criteria of triviality/linearizability.
The neighborhood of a rational curve $C_0$ in a projective surface $S$ is always linear
(see Proposition \ref{prop:ProjectiveSurface}). As shown by Arnol'd, if the local deformations 
of $C_0$ are the geodesics of a projective structure on $U^\vee$, then we are again in the linear case.
In fact, in the non linear case, it is shown in Proposition \ref{prop:2pencils} that deformations 
$C_\epsilon$ of $C_0$ passing through a general point 
$p$ of $(U^\vee,C_0)$ are only defined for $\epsilon$ close to $0$: there is no local {\bf pencil} of smooth analytic curves
through $p$ that contains the germs $C_\epsilon$ at $p$. We show in Proposition \ref{prop:2pencils}
that, in the non linear case, there is at most one point $p$ where we get such pencil.

Going back to real analytic metrics, the three geometries of Klein, considering metrics
of constant curvature, give birth to the same (real) projective structure, namely the trivial one. 
Indeed, geodesics of the unit $2$-sphere $\mathbb S^2\subset\mathbb R^3$ are defined as intersections
with planes passing through the origin: they project on lines, from the radial projection to a general affine plane.
Similarly, for negative curvature, geodesics are lines in Klein model. It would be nice to understand
{\it which $(+1)$-neighborhoods $(U^\vee,C_0)$ come from the geosedics of a holomorphic metric}. 

In section \ref{Sec:FlatFibration}, we introduce the notion of {\bf foliated projective structure}, when 
the projective structure is defined by a flat affine connection $\nabla$, i.e. satisfying $\nabla\cdot\nabla=0$.
Equivalently, the collection of geodesics decomposes as a pencil of geodesic foliations. 
On the dual picture $(U^\vee,C_0)$, such a decomposition corresponds to an analytic {\bf fibration}
transversal to $C_0$, i.e. a holomorphic retraction $U^\vee\to C_0$. This dictionary appear
in Kry\'nski \cite{Krynski}. Our main result, announced in \cite{CRAS}, is that non linear 
$(+1)$-neighborhoods $(U^\vee,C_0)$ have $0$, $1$ or $2$ transverse fibrations, no more
(see Theorem \ref{3-fibrations}). We show that each case occur with an infinite dimensional 
moduli space. 

The main ingredient to study the existence and unicity of transverse fibrations is the classification
of $(+1)$-neighborhoods up to analytic isomorphisms, which is due to Mishustin \cite{Mishustin} (section \ref{sec:ClassifMishNeigh}). 
It was known since the work of 
Grauert \cite{Grauert} that there are infinitely many obstructions to linearize such a neighborhood.
 Hurtubise and Kamran \cite{Hurtubise} provided a normal form for the formal neighborhood,
and one year later, 
Mishustin showed the convergence of that normal form by a different and geometrical approach 
in \cite{Mishustin}. Precisely,
any positive neighborhood can be described as the patching of two open sets of
the linear neighborhood by a non linear cocycle, that can be reduced to an almost unique normal form
(Theorem \ref{Mishustin} and Proposition \ref{PropMishustinFreedom}).
The moduli space appears to be isomorphic to the space of convergent power series in two variables.
Hurtubise and Kamran \cite{Hurtubise} also provide explicit formulae linking these invariants 
(coefficients of the cocycle) with Cartan invariants for the equivalence problem for projective structures 
(or second order differential equations). It is quite surprising that these two works have never been quoted until our announcement \cite{CRAS}, 
although it answers a problem left opened since the celebrating works of Grauert and Kodaira. 
In Proposition \ref{PropMishustinFreedom}, we get a more precise description
of the freedom in the reduction to normal form which is necessary for our purpose,
namely the action of a $4$-dimensional linear group (see Corollary \ref{Cor:ActionLinGroup}).

From Mishustin's cocycle (and its non unicity), we see in Proposition \ref{prop:4NeighbFibr}  
that the first obstruction to the existence to a transverse fibration
arise in $5$-jet, i.e. in the $5^{\text{th}}$ infinitesimal neighborhood of the rational curve, which was
also surprising for us. Another surprising fact is the existence of many neighborhoods with two 
fibrations: we get a moduli space (Theorem \ref{thm:TangGenrericModuli}) isomorphic to the space of power series in one variable.
One remarkable example (see section \ref{tancency-along-C_0}) is given by the two-fold ramified covering 
$(U^\vee,C_0)\stackrel{2:1}{\to}(\mathbb P^1\times\mathbb P^1,\Delta)$
of a neighborhood of the diagonal 
$\Delta\subset\mathbb P^1\times\mathbb P^1$ that ramifies precisely along $\Delta$: 
the two fibrations of $\mathbb P^1\times\mathbb P^1$ lift as fibrations tangent all along $C_0$. 
This example is non linear, and in particular non algebraic (the covering can be only defined 
at the neighborhood of $\Delta$ for topological reasons). However, the field of meromorphic
functions on $(U^\vee,C_0)$ identifies with the field of rational functions on $\mathbb P^1\times\mathbb P^1$
and has transcendance dimension $2$. We expect that the general $(+1)$-neighborhood has no meromorphic
function, but we have no proof, and no example. We are able to compute the differential equation 
defining the dual projective structure, namely $y''=(xy'-y)^3$. This example is also remarkable because
it has the largest symmetry group, namely $\mathrm{SL}_2(\mathbb C)$, and this is an ingredient of the proof.

The paper is mostly a survey putting together some results that were scattered in the litterature.
There are however original results as the complete proof of Theorem \ref{3-fibrations}
annouced in \cite{CRAS}, the precise description of non unicity for Mishustin cocycle given 
in Proposition \ref{PropMishustinFreedom}, the meromorphic propagation of foliated structures
in Lemma \ref{Lem:ProlongationFolAD0} and finally the application to Painlev\'e equations 
in section \ref{sec:Painleve}. Our work is used in the recent publication \cite{DunajskiWaterhouse}.

\section{Projective structure, geodesics and duality}\label{Sec:duality}

\subsection{Second order differential equations and duality}

Let $(x,y)$ be coordinates of $\mathbb C^2$.
Given a $2^{\text{nd}}$ order differential equation 
\begin{equation}\label{eq:2ndODE}
y''=f(x,y,y')
\end{equation}
with $f(x,y,z)$ holomorphic at the neighborhood $V$ of some point $(0,0,z_0)$ say,
local solutions $y(x)$ lift as Legendrian curves for the contact structure defined by
$$\alpha=0\ \ \ \text{where}\ \ \ \alpha=dy-zdx.$$
We get two transversal Legendrian foliations on $V$. The first one $ \mathcal{F} $
is defined by the fibers of the projection $V\to U;(x,y,z)\mapsto(x,y)$.
The second one $\mathcal G$ is defined by solutions $x\mapsto(x,y(x),y'(x))$
or equivalently by trajectories of the vector field 
$$v=\partial_x+z\partial_y+f(x,y,z)\partial_z.$$
More generally, given a germ of contact $3$-manifold together with two transversal 
Legendrian foliations, the space of $\mathcal F$-leaves can be identified with an
open set $U\subset\mathbb C^2$ with coordinates $(x,y)$ and $\mathcal G$-leaves
project on $U$ as graphs of solutions of a $2^{\text{nd}}$ order differential equation
$y''=f(x,y,y')$, see \cite[Chapter 1, Section 6.F]{Arnold}.

It is now clear that the role of $\mathcal F$ and $\mathcal G$ can be 
permuted: {\it on the space $U^\vee$ of $\mathcal G$-leaves, $\mathcal F$-leaves project
to solutions of a $2^{\text{nd}}$ order differential equation $Y''=g(X,Y,Y')$}
(once we have choosen coordinates $(X,Y)\in U^\vee$). This is the duality introduced
by Cartan (see also \cite[Chapter 1, Sections 6.F, 6.G]{Arnold}).
Points on $U$ correspond to curves on $U^\vee$ and {\it vice-versa}.
We will call $V$ the incidence variety by analogy with the case of lines in $\mathbb P^2$.

For instance, lines $y=ax+b$ are solutions of the differential equation $y''=0$. Using $(X,Y)=(a,b) \in \Pd$ for coordinates of dual points, we see that foliations $\FF$ and $\GG$ given before are liftings of lines on the projective and dual plane, thus the dual  equation is also $Y''=0$.

If there is a diffeomorphism $\phi:U\to\tilde U$ sending solutions of the differential
equation to the solutions of another one $y''=\tilde f(x,y,y')$ on $\tilde U$,
then $\phi$ can be lifted to a diffeomorphism $\Phi:V\to \tilde V$ conjugating 
the pairs of Legendrian foliations: $\Phi_*\mathcal F=\tilde{\mathcal F}$ and 
$\Phi_*\mathcal G=\tilde{\mathcal G}$. We say that the two differential 
equations are {\it Cartan-equivalent} in this case.

\subsection{Projective structure and geodesics}

When the differential equation (\ref{eq:2ndODE}) is cubic in $y'$
$$y''=A(x,y)+B(x,y)(y')+C(x,y)(y')^2+D(x,y)(y')^3$$
(where $A,B,C,D$ are holomorphic functions on $U$), then the foliation $\mathcal G$
is global on $V:=\mathbb P(T_U)\simeq U\times\mathbb P^1_z$, $z=\frac{dy}{dx}$,
and transversal to the fibration $\mathcal F$
everywhere. Precisely, setting $\tilde z=\frac{1}{z}=\frac{dx}{dy}$, then 
the foliation $\mathcal G$ is defined by the two vector field
$$v=\partial_x+z\partial_y+(A+Bz+Cz^2+Dz^3)\partial_z$$
for $z$ finite, and 
$$\tilde v=\tilde z\partial_x+\partial_y-(D+C\tilde z+B\tilde z^2+A\tilde z^3)\partial_{\tilde z}$$
near $z=\infty$. 

\begin{remark}For equations $y''=f(x,y,y')$ having right-hand-side $f(x,y,y')$ polynomial with respect to $y'$, 
but higher than cubic degree, the foliation $\mathcal G$ globalizes on $U\times\mathbb C_z$ 
but transversality is violated at $z=\infty$. Indeed, the corresponding vector field 
$$\tilde v=\tilde z\partial_x+\partial_y-\tilde z^3 f\left(x,y,\frac{1}{\tilde z}\right)\partial_{\tilde z}$$
becomes meromorphic; after multiplication by a convenient power of $\tilde z$,
the vector field becomes holomorphic but tangent to $\mathcal F$ along $\tilde z=0$, and its trajectories
become singular after projection on $U$.
\end{remark}

With the previous remark, it is easy to check that any foliation $\mathcal G$ on $\mathbb P(T_U)$
which is 
\begin{itemize}
\item Legendrian, i.e. tangent to the natural contact structure ($dy-zdx=0$),
\item transversal to the projection $\mathbb P(T_U)\to U$,
\end{itemize} 
is locally defined by a vector field like above, cubic in $z$, i.e. by a second order differential equation
with $y''=A+B(y')+C(y')^2+D(y')^3$. We call {\it projective structure} such a data.
We call {\it geodesic} a curve on $U$ obtained by projection of a $\mathcal G$-leaf on $\mathbb P(T_U)$.
The following is proved by Le Brun in \cite[Section 1.3]{LeBrun}

\begin{prop}\label{prop:convexity}
If $U$ is a sufficiently small ball,
then all geodesics are properly embedded discs and we have the following properties:
\begin{itemize}
\item \emph{convexity:} through any two distinct points $p,q\in U$ passes a unique geodesic;
\item \emph{infinitesimal convexity:} through any point $p\in U$ and in any direction $l\in T_U\vert_p$
passes a unique geodesic.
\end{itemize}
We say that $U$ is geodesically convex in this case.
\end{prop}

The second item just follows from Cauchy-Kowalevski Theorem for the differential equation 
defining the projective structure.

\subsection{Space of geodesics and duality}

It is proved in \cite[Section 1.4]{LeBrun} the following

\begin{prop}If $U$ is geodesically convex, then the space of geodesics,
i.e. the quotient space 
$$U^\vee:=\mathbb P(T_U)/\mathcal G$$
is a smooth complex surface. Moreover, the projection map
$$\pi^\vee:\mathbb P(T_U)\to U^\vee$$
restricts to fibers of $\pi:\mathbb P(T_U)\to U$ as an embedding.
\end{prop}

We thus get a two-parameter family (parametrized by $U$) of smooth rational curves 
covering the surface $U^\vee$: for each point $p\in U$, we get a curve $C_p\subset U^\vee$, namely $\pi^\vee(\pi^{-1}(p))$. 
The curve $C_p$ parametrizes in $U^\vee$ the set (pencil) of geodesics passing through $p$.
Any two curves $C_p$ and $C_q$, with $p\not=q$, intersect transversely through a single point in $U^\vee$
representing the (unique) geodesic passing through $p$ and $q$.
The normal bundle of any such curve $C_p$
is in fact $\mathcal O_{\mathbb P^1}(1)$ (after identification $C_p\simeq \mathbb P^1$).

One might think that rational curves define the geodesics of a projective structure on $U^\vee$,
but it is almost never true: for instance, the set of rational curves (of the family $C_p$)
through a given point of $U^\vee$ cannot be completed as a pencil of curves
(as it would be for geodesics of a projective structure), see \cite[Chapter 1, Section 6-D]{Arnold}. 
In fact, we will prove (see Proposition \ref{prop:2pencils}) that, if such a pencil exists at two different points of $U^\vee$,
then we are essentially in the standard linear case of lines in $\mathbb P^2$.

\begin{figure}[ht!]
  \centering
    \includegraphics[scale=0.5]{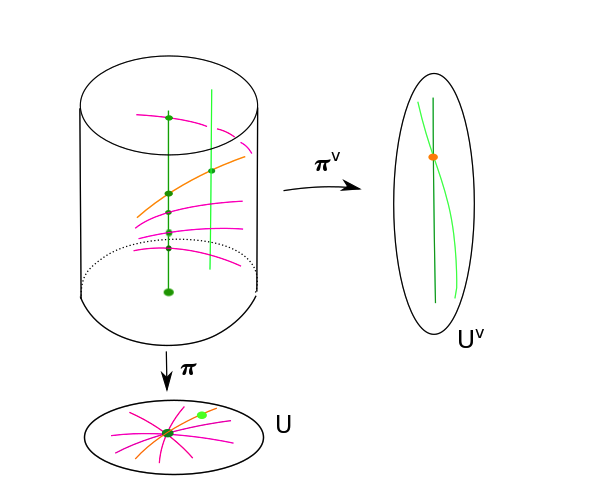}
  \caption{Duality picture}
  \label{fig:bundlepicture}
\end{figure}

From a germ of projective structure at $p\in U$, we can deduce a germ of surface
neighborhood of $C_p\simeq\mathbb P^1$. Conversely, it is proved in 
\cite[Section 1.7]{LeBrun} that we can reverse the construction. Indeed, given a rational
curve $C\subset S$ in a surface (everything smooth holomorphic) having normal bundle 
$\mathcal O_{\mathbb P^1}(1)$, then $C$ admits by Kodaira Deformation Theory
a local $2$-parameter family of deformation and the parameter space $U$
is naturally equipped with a projective structure: geodesics on $U$ are those
rational curves passing to a common point in $S$.

In the sequel, we call $(+1)$-neighborhood of a rational curve $C$
a germ $(S,C)$ of a smooth complex surface $S$ where $C$ is embedded
with normal bundle $N_C\simeq \mathcal O_C(+1)$.

\begin{thm}[Le Brun] We have a one-to-one correspondance between germs
of projective structures on $(\mathbb C^2,0)$ up to diffeomorphism and 
germs of $(+1)$-neighborhood of $\mathbb P^1$ up to isomorphism.
\end{thm}

More details (and generalizations) can be found in \cite{LeBrun} (see also \cite{Hitchin0}).

\subsection{Affine connections, metric}
Let $S$ be a smooth complex surface. An \textit{affine connection} on $S$ is 
a (linear) holomorphic connection on the tangent bundle $T_S$, i.e. 
a $\mathbb{C}$-linear morphism $\nabla : T_S \rightarrow T_S \otimes \Omega^1_S$ satisfying the Leibnitz rule 
$$
\nabla(f\cdot Z)= Z \otimes df + f\cdot \nabla(Z)
$$
for any holomorphic function $f$ and any vector field $Z$. Given a two vector fields $Z,W$,
we denote as usual by $\nabla_W Z:=i_W(\nabla Z)$ the contraction of $\nabla Z$ with $W$.
By a \textit{parametrized geodesic} for $\nabla$, we mean a holomorphic curve $t\mapsto\gamma(t)$ on $S$ 
such that $\nabla_{\dot\gamma(t)}\dot\gamma(t) =0$ on the curve. The image of $\gamma(t)$ on $S$ 
is simply called 
a \textit{ (unparametrized) geodesic} and is characterized by $\nabla_{\dot\gamma(t)}\dot\gamma(t) =f(t)\dot\gamma(t)$ for any parametrization. Geodesics define a projective structure $\Pi_\nabla$ on $S$.

In coordinates $(x,y)\in U\subset\mathbb C^2$, a trivialization of $TU$ is given by the basis $(\partial_x,\partial_y)$
and the affine connection is given by 
$$\nabla(Z)=d(Z)+\Omega\cdot Z,\ \ \ \Omega = \left( \begin{matrix}
\alpha & \beta\\
\gamma & \delta
\end{matrix} \right)$$
where $Z=z_1\partial_x+z_2\partial_y$ and $\alpha,\beta,\gamma,\delta\in\Omega^1(U)$.
On the projectivized bundle $\mathbb P(T_U)$, with trivializing coordinate $z=z_2/z_1$,
equation $\nabla=0$ induces a \textit{Riccati distribution}
\begin{equation}\label{eq:RiccatiDistrib}
\ker(dz+\gamma-(\alpha-\delta)z-\beta z^2).
\end{equation}
Intersection with the contact structure $\ker(dy-zdx)$ gives the geodesic foliation $\mathcal G$
of the projective structure. Precisely, if we set
\begin{equation}\label{eq:Expandxdy}
\left( \begin{matrix}
\alpha & \beta\\
\gamma & \delta
\end{matrix} \right)
=
\left( \begin{matrix}
\alpha_1 & \beta_1\\
\gamma_1 & \delta_1
\end{matrix} \right)dx
+
\left( \begin{matrix}
\alpha_2 & \beta_2\\
\gamma_2 & \delta_2
\end{matrix} \right)dy
\end{equation}
(with $\alpha_i,\beta_i,\gamma_i,\delta_i\in\mathcal O(U)$) then the projective structure is given by substituting
(\ref{eq:Expandxdy}) and $z=\frac{dy}{dx}$ into (\ref{eq:RiccatiDistrib}), namely
\begin{equation}\label{eq:FromRiccToProj}
\frac{d^2y}{dx^2}=\underbrace{-\gamma_1}_A+\underbrace{(\alpha_1-\delta_1-\gamma_2)}_B\left(\frac{dy}{dx}\right)+\underbrace{(\beta_1+\alpha_2-\delta_2)}_C\left(\frac{dy}{dx}\right)^2+ \underbrace{\beta_2}_D \left(\frac{dy}{dx}\right)^3.
\end{equation}
We say that two affine connections are (projectively) equivalent if they have the same family of geodesics,
i.e. if they define the same projective structure. The following is straightforward

\begin{lemma}\label{affine-connections-defining-the-same-projective}
Two affine connections $\nabla$ and $\nabla'$ on $U$, with matrices $\Omega$ and $\Omega'$ respectively, define the same projective structure if, and only if, there are $a,b,c,d \in \OO(U)$ such that 
$$
\Omega' = \Omega + a \left(\begin{matrix} 
dx /2 & 0 \\
dy & -dx /2
\end{matrix} \right)
 + b \left(\begin{matrix} -dy /2 & dx \\
0 & dy /2\end{matrix} \right) + c \left(\begin{matrix} dx  & 0 \\
0 & dx \end{matrix} \right) + d \left(\begin{matrix} dy  & 0 \\
0 & dy\end{matrix} \right).
$$
\end{lemma}

\begin{remark}Any projective connection $\Pi : y'' = A + B(y') + C(y')^2 +D(y')^3$
can be defined by an affine connection: for instance, $\Phi=\Pi_\nabla$ with
$$\nabla=d+\left( \begin{matrix}
0 & Cdx+Ddy\\
-Adx-Bdy & 0
\end{matrix} \right)$$
or is equivalently defined by the Riccati distribution
$$ dz+Adx+Bdy+z^2(Cdx+Ddy)=0.$$

\end{remark}

One can also define a projective structure by a holomorphic Riemannian metric, 
by considering its geodesics defined by Levi-Civita (affine) connection.
But it is not true that all projective structures come from a metric:
in \cite{BDE}, it is proved that there are infinitely many obstructions, the first one
arising at order $5$.

\begin{question}Can we characterize in a geometric way those projective structures
arising from a holomorphic metric ? And what about the corresponding $(+1)$-neighborhood ?
\end{question}

\subsection{Some criteria of linearization}\label{ssec:criteriaLin}

A projective structure $(U,\Pi)$ is said {\it linearizable} if it is Cartan-equivalent
to the standard linear structure whose geodesics are lines: there is a diffeomorphism
$$\Phi:U\to V\subset\mathbb P^2$$
such that geodesics on $U$ are pull-back of lines in $\mathbb P^2$.
When $U$ is geodesically convex, this is equivalent to say that $(U^\vee,C_0)$ 
is the neighborhood of a line in $\mathbb P^2$.
As we shall see later, there are many projective structures that are non linearizable (even locally).
Here follow some criteria of local linearizability.

\begin{prop}\label{prop:2pencils}
Let $\Pi$ be a germ of projective structure at $(\mathbb C^2,0)$ and let 
$(U^\vee,C_0)$ be the corresponding $(+1)$-neighborhood. If for $2$ distinct points
$p_1,p_2\in C_0$ the family of rational curves through $p_i$ is contained in a pencil
of curves based in $p_i$, then $(U^\vee,C_0)\simeq(\mathbb P^2,\text{line})$ (and $\Pi$ is 
linearizable).
\end{prop}

\begin{proof}For $i=1,2$, let $F_i:U^\vee\dashrightarrow\mathbb P^1$ be the meromorphic map 
defining the pencil based at $p_i$: deformations of $C_0$ passing through $p_i$ 
are (reduced) fibers of $F_i$. We can assume $C_0=\{F_i=0\}$ for $i=1,2$.
Then, maybe shrinking $U^\vee$,  the map
$$\Phi:U^\vee\to \mathbb P^2_{(z_0:z_1:z_2)}\ ;\ p\mapsto(1:\frac{1}{F_1}:\frac{1}{F_2})$$
is an embedding of $U^\vee$ onto a neighborhood of the line $z_0=0$.
Indeed, $\Phi$ is well-defined and injective on $U^\vee\setminus C_0$;
one can check that it extends holomorphically on $C_0$ and the extension does not contract this curve.\end{proof}

\begin{cor}{\cite[Chapter 1]{Arnold}} Let $\Pi$ be a germ of projective structure at $(\mathbb C^2,0)$ and let 
$(U^\vee,C_0)$ be the corresponding $(+1)$-neighborhood. If deformations of $C_0$ are geodesics
of a projective structure $\Pi^\vee$  in a neighborhood of a point $p\in C_0$, 
then $(U^\vee,C_0)\simeq(\mathbb P^2,\text{line})$ (and $\Pi$ is 
linearizable).
\end{cor}

\begin{prop}[{\cite[Proposition 4.7]{Hurtubise}}]\label{prop:ProjectiveSurface}
Let $S$ be a smooth compact surface with an embedded curve $C_0 \simeq \mathbb{P}^1$ with self-intersection $+1$. Then $S$ is rational and $(S,C_0) \simeq (\mathbb{P}^2, line)$.
\end{prop}

\begin{proof}
As $S$ contains a smooth rational curve with positive self-intersection, we deduce from \cite[Proposition V.4.3]{BPV} that $S$ is rational. This implies $H^1(S, \OO_S) \simeq H^{0,1}(S) \simeq H^{1,0}(S) \simeq H^0(S, \Omega^1_S)=0$ thus the Chern-class morphism $H^1(S, \OO^*_S) \rightarrow H^2(S, \mathbb{Z})$ is injective. We can take deformations $C_1$, $C_2$ of $C_0$ such that $C_0 \cap C_1 \cap C_2 = \emptyset$ 
and by the previous discussion the three curves determine the same element $\OO_S(C)$ of $H^1(S, \OO^*_S)$, then we have sections $F_i$ of $\OO_S(C)$ vanishing on $C_i$, $i=0,1,2$. We define
$$
\sigma:=(F_0:F_1:F_2): S \dashrightarrow \PP
$$
which is in fact a morphism. Moreover, by the condition on the intersection of the curves we deduce that the generic topological degree of $\sigma$ is $1$. In particular $\sigma$ is a sequence of blow-ups with no exceptional divisor intersecting $C_0$.
\end{proof}

\begin{prop}
There is a unique global projective structure on $\PP$, namely the linear one.
\end{prop}

\begin{proof}
If $\FF_{\Pi}$ is the associated regular foliation by curves defined in $M = \Pp(T\PP)$ with cotangent bundle $\OO_M(ah + b \check{h})$ and $\VV$ stands for the foliation defined by the fibers, then $Tang(\FF_{\Pi}, \VV) = (a+2)h + (b-1)\check{h}$ (see \cite[Proposition 2.3]{MF}). So, the only second order differential equation totally transverse to $\VV$  is the one given by $y''= 0$.
\end{proof}

Finally, we end-up with an analytic criterium proved by Liouville in \cite{Liouville} (and later by Tresse and Cartan):
\begin{prop}\label{Equations-L1-L2}
Given a projective structure $\Pi$ defined by equation (\ref{eq:ProjStructDiffEq}) in the introduction, then 
consider the two functions $L_i(x,y)$ defined by
\begin{equation}\label{eq:Liouville}
\left\{\begin{matrix}
L_1=2B_{xy}-C_{xx} -3A_{yy}-6AD_x -3A_xD +3(AC)_y  +BC_x - 2BB_y,\\
L_2=2C_{xy}-B_{yy} -3D_{xx}+6A_yD +3AD_y  -3(BD)_x -B_yB + 2CC_x.
\end{matrix}\right.
\end{equation}
Then, the tensor $\theta:=(L_1 dx + L_2dy)\otimes ( dx \wedge dy)$ is intrinsically
defined by the projective structure $\Pi$, i.e. its definition does not depend on the choice of coordinates $(x,y)$.
Moreover, $\Pi$ is linearizable if, and only if, $\theta \equiv 0$.
\end{prop}

We promtly deduce:

\begin{cor}\label{prop:ConnectedLinearization}
Let $\Pi$ be a projective structure on a connected open set $U$.
If $\Pi$ is linearizable at the neighborhood of a point $p\in U$, then it is linearizable
at the neighborhood of any otherpoint $q\in U$.
\end{cor}

\section{Foliated structures vs transverse fibrations}\label{Sec:FlatFibration}

A (non singular) foliation $\mathcal F$ on $U$, defined by say $y'=f(x,y)$,
can be equivalently defined by its graph $S:=\{z=f(x,y)\}\subset \mathbb P(T_U)$, 
a section of $\pi:\mathbb P(T_U)\to U$. The foliation is geodesic if, and only if, the section $S$
is tangent to $\mathcal G$; in this case, the section projects onto a curve $D:=\pi^\vee(S)$  
intersecting transversally the rational curve $C_0$ at a single point on $U^\vee$.
We thus get a one-to-one correspondance between geodesic foliations on $U$ 
and transversal curves on $(U^\vee,C_0)$.

\subsection{Pencil of foliations and Riccati foliation}
A (regular) \textit{pencil of foliations} on $U$ is a one-parameter family of foliations $\{ \FF_{t} \}_{t \in \Pp^{1}}$
defined by $\FF_t = \ker(\omega_t)$ for a pencil of $1$-forms $\{ \omega_t=\omega_0+t\omega_\infty \}_{t \in \Pp^{1}}$ with 
$\omega_0,\omega_\infty\in\Omega^1(U)$ and $\omega_0\wedge\omega_\infty\not=0$ on $U$.
The pencil of $1$-forms defining $\{ \FF_{t} \}_{t \in \Pp^{1}}$ is unique up to multiplication
by a non vanishing function:
$\tilde\omega_t=f\omega_t$ for all $t\in \Pp^{1}$ and $f\in\mathcal O^\vee(U)$.
In fact, the parametrization by $t\in \Pp^{1}$ is not intrinsec; we will say
that $\{ \FF_{t} \}_{t \in \Pp^{1}}$ and $\{ \FF'_{t} \}_{t \in \Pp^{1}}$ define
the same pencil on $U$ if there is a Moebius transformation $\varphi\in\mathrm{Aut}(\mathbb P^1)$
such that $\FF'_{t}=\FF_{\varphi(t)}$ for all $t\in \Pp^{1}$.

There exists a unique projective structure $\Pi$ whose geodesics are the leaves of the pencil. 
Indeed, the graphs $S_t$ of foliations $\FF_t$ are disjoint sections (since foliations are pairwise transversal)
and form a codimension one foliation $\mathcal H$ on $\mathbb P(T_U)$ transversal
to the projection $\pi:\mathbb P(T_U)\to U$. The foliation $\mathcal H$ is a Riccati foliation,
i.e. a Frobenius integrable Riccati distribution:
$$\mathcal H=\ker(\omega),\ \ \ \omega=dz+\alpha z^2+\beta z+\gamma,\ \ \ \omega\wedge d\omega=0.$$
Intersecting $\mathcal H$ with the contact distribution
yields a Legendrian foliation $\mathcal G$ (also transversal to the $\mathbb P^1$-fibers)
and thus a projective structure.

In local coordinates $(x,y)$ such that $\mathcal F_0$ and $\mathcal F_\infty$
are respectively defined by $\ker(dx)$ and $\ker(dy)$, we can assume the pencil generated by
$\omega_0=dx$ and $\omega_\infty=u(x,y)dy$ (we have normalized $\omega_0$)
with $u(0,0)\not=0$. Then, the graph of the foliation $\FF_t$ is given by the section
$S_t=\{z=-\frac{1}{tu(x,y)}\}\subset\mathbb P(T_U)$. These sections are the leaves 
of the Riccati foliation $\mathcal H=\ker(dz+\frac{du}{u}z)$, and we can deduce the equation of the projective structure:
$$y''+\frac{u_x}{u} y'+\frac{u_y}{u} (y')^2=0.$$
Note that the projective structure is also defined by the affine connection
$$\nabla=d+\Omega,\ \ \ \Omega=\begin{pmatrix}\frac{1}{2}\frac{du}{u}&0\\0&-\frac{1}{2}\frac{du}{u}\end{pmatrix}$$
which is flat (or integrable, curvature-free) $\Omega\wedge d\Omega=0$, and trace-free
$\mathrm{trace}(\Omega)=0$. 

\begin{remark}A Riccati distribution $\mathcal H=\ker(\omega)$  on $\mathbb P(T_U)$,
$$\omega=dz+\alpha z^2+\beta z+\gamma,\ \ \ \alpha,\beta,\gamma\in\Omega^1(U),$$
is the projectivization of a unique trace-free affine connection, namely
$$\nabla=d+\Omega,\ \ \ 
\Omega=\begin{pmatrix}-\frac{\beta}{2}&-\alpha\\ \gamma&\frac{\beta}{2}\end{pmatrix}.$$
Are equivalent
\begin{itemize}
\item $\nabla$ is flat: $\Omega\wedge\Omega+d\Omega=0$;
\item $\omega$ is Frobenius integrable: $\omega\wedge d\omega=0$.
\end{itemize}
There are many other affine connections whose projectivization is $\omega$
which are not flat: in general, we only have the implication [$\nabla$ flat] $\Rightarrow$ [$\mathcal H$ integrable].
\end{remark}

\subsection{Transverse fibrations on $U^\vee$}\label{sec:TransvFibration}

If we have a Riccati foliation $\mathcal H$ on $\mathbb P(T_U)$ which is $\mathcal G$-invariant,
then it descends as a foliation $\underline{\mathcal H}$ on $U^\vee$ transversal to $C_0$.
Maybe shrinking $U^\vee$, we get a fibration by holomorphic discs transversal to $C_0$
that can be defined by a holomorphic submersion 
$$H:U^\vee\to C_0$$
satisfying $H\vert_{C_0}=id\vert_{C_0}$ (a holomorphic retraction). 
Indeed, we can define this map on $\mathbb P(T_U)$ first (construct a first integral for $\mathcal H$)
and, then descend it to $U^\vee$.

Conversely, if we have a holomorphic map $H:U^\vee\to \mathbb P^1$ which is a submersion
in restriction to $C_0$, then fibers of $\tilde H:=H\circ \pi^\vee:\mathbb P(T_U)\to\mathbb P^1$
define the leaves of a Riccati foliation $\mathcal H$. Indeed, the restriction of $\tilde H$
to $\mathbb P^1$-fibers must be global diffeomorphisms, and in coordinates, 
$\tilde H$ take the form
$$\tilde H(x,y,z)=\frac{a_1(x,y)z+b_1(x,y)}{a_2(x,y)z+b_2(x,y)}$$
which, after derivation, give a Riccati distribution:

$$
\ker(d\tilde H)=$$
$$\ker\left(dz+\frac{a_2 da_1-a_1 da_2}{a_1 b_2-b_1 a_2}z^2+\frac{a_2 db_1-b_1 da_2+b_2 da_1-a_1 db_2}{a_1 b_2-b_1 a_2}z
+\frac{b_2 db_1-b_1 db_2}{a_1 b_2-b_1 a_2}\right).$$
By construction, the Riccati foliation $\mathcal H$ is $\mathcal G$-invariant.
One therefore deduce:

\begin{prop}[Krynski  \cite{Krynski}]\label{pencil-vs-flat-vs-fibration}Let $\Pi$ be a projective structure on $(U,0)$
and $(U^\vee,C_0)$ be the dual. The following data are equivalent:
\begin{itemize}
\item a pencil of geodesic foliations $\{ \FF_{t} \}_{t \in \Pp^{1}}$,
\item a $\mathcal G$-invariant Riccati foliation $\mathcal H$ on $\mathbb P(T_U)$,
\item a fibration by discs transversal to $C_0$ on $U^\vee$.
\end{itemize}
\end{prop}

In this case, we say that the projective structure is \textit{foliated}.

\begin{example}\label{conexion-trivial}
Let $\Pi_0$ be the trivial structure $y ''=0$ with Riccati distribution $\omega_0 = dz$. 
We fix coordinates $(a,b) \in \check{\Pp}^2$ for the line $\{ax+by = 1\} \subseteq \PP$ and observe that $\check{0}$ is the line of the infinity
$L_\infty$ in this coordinates. It is straightforward to see 
that the Riccati foliation associated to the pencil of lines through $(a,b)$ is
$$
\omega = dz + \left( \left( \frac{a}{1-ax-by}\right) + z \left( \frac{b}{1-ax-by}\right) \right)(dy - zdx). 
$$
Remark that the fibrations induced by $\omega_0$ and $\omega$ have a common fiber, which is the fiber associated to the radial foliation with center at $\{ax+by=1\} \cap L_{\infty}$.
\end{example}

\subsection{Webs and curvature} 
We recall that a (regular) $k$-web $\WW=\FF_1 \boxtimes \ldots \boxtimes \FF_{k}$ on $(\CC^2,0)$ is a collection of $k$ germs of (smooth) pair-wise transversal foliations. We say that the projective structure $\Pi$ is compatible 
with a web $\WW$ if every leaf of $\WW$ is a geodesic of $\Pi$. 
For $4$-webs we have the following proposition.

\begin{prop}[{\cite[Proposition 6.1.6]{PereiraPirio}}]\label{unique-structure-compatible-4-web}
If $\WW$ is a (regular) $4$-web on $(\CC^2,0)$ then there is a unique projective structure $\Pi_{\WW}$ compatible with $\WW$. 
\end{prop}

Let $\WW=\FF_1 \boxtimes \FF_2 \boxtimes \FF_3 \boxtimes \FF_4$ be a $4$-web on $(\CC^2,0)$
$$\FF_i = [X_i = \partial_x + e_i (x,y) \partial_y] = [\eta_i = e_i dx - dy],\ \ \ i=1,2,3,4.$$
The cross-ratio 
\begin{equation}\label{eq:CrossRatioFol}
\left(\FF_1,\FF_2;\FF_3,\FF_4\right):=\frac{(e_1-e_3)(e_2-e_4)}{(e_2-e_3)(e_1-e_4)}
\end{equation}
is a holomorphic function on $(\CC,0)$ intrinsically defined by $\WW$. 
Then, we have:

\begin{prop}\label{unique-pencil-3-web}
If $\WW=\FF_0 \boxtimes \FF_1 \boxtimes \FF_{\infty}$ is a regular $3$-web on $(\CC^2,0)$,
then there is a unique pencil $\{\FF_t\}_{t\in\mathbb P^1}$ that contains $\FF_0$,  $\FF_1$
and $\FF_{\infty}$ as elements. Precisely, $\FF_t$ is defined as the unique foliation such that
$$\left(\FF_0,\FF_1;\FF_t,\FF_\infty\right)\equiv t.$$
We denote by  $\Pi_{\WW}$ the corresponding projective structure on $(\CC ^{2},0)$.
\end{prop}

Conversely, any foliated projective structure comes from a $3$-web: it suffices to choose
$3$ elements of a pencil. In particular, any $4$ elements of a pencil $\{\FF_t\}_{t\in\mathbb P^1}$
have constant cross-ratio. 

\begin{remark}\label{rem:LinearClosedHexag}
The projective structure $\Pi_{\WW}$ is linearizable if, and only if, 
the pencil $\{\FF_t\}$ can be defined by a pencil of closed $1$-forms (see \cite{Nakai} or \cite{LinsNeto}).
Equivalently, any extracted $3$-web is hexagonal
(i.e. hexagonal, see \cite[section 6]{BeLiLo} or \cite[Chap. 1, Sect. 2]{PereiraPirio}).
\end{remark}

\begin{example}
We can easily construct non linearizable projective structure by using Remark \ref{rem:LinearClosedHexag}.
For instance, the projective structure generated by the pencil of $1$-forms $\omega_t:=dx+te^{xy}dy$ 
cannot be defined by a pencil of closed $1$-forms. Therefore, it defines a non linearizable pencil $\WW$,
and a non linearizable projective structure $\Pi_{\WW}$.
\end{example}

\subsection{Prolongation of foliated structures}

Let $\Pi$ be a projective structure defined on an open set $\Omega\subset\CC^2$ by the differential equation
\begin{equation}\label{eq:ProjStructDiffEqBis}
y''=A(x,y)+B(x,y)(y')+C(x,y)(y')^2+D(x,y)(y')^3
\end{equation}
We assume that $\Pi$ is foliated: it can be defined from a Riccati foliation on $\PP T\Omega$, i.e. a Pfaffian equation
\begin{equation}\label{eq:Riccati}
dz=\alpha z^2+\beta z+\gamma
\end{equation}
where Frobenius integrability condition writes:
\begin{equation}\label{eq:FrobeniusRicc}
\left\{\begin{matrix} d\alpha=&\alpha\wedge\beta\\
d\beta=&2\alpha\wedge\gamma\\
d\gamma=&\beta\wedge\gamma
\end{matrix}\right.
\end{equation}
Substituting $z=\frac{dy}{dx}$ in (\ref{eq:Riccati}) and comparing with (\ref{eq:ProjStructDiffEqBis}) yields
$$y''=\gamma_1+(\beta_1+\gamma_2)(y')+(\alpha_1+\beta_2)(y')^2+\alpha_2(y')^3$$
with
\begin{equation}\label{eq:FromProjToRicc}
\left\{\begin{matrix} 
\alpha=\alpha_1 dx+\alpha_2 dy=&(C-g)dx+D dy\\
\beta=\beta_1 dx+\beta_2 dy=&fdx+gdy\\
\gamma=\gamma_1dx+\gamma_2 dy=&A dx+(B-f)dy
\end{matrix}\right.
\end{equation}
By substituting (\ref{eq:FromProjToRicc}) in (\ref{eq:FrobeniusRicc}), we promptly deduce:

\begin{prop}\label{Prop:FoliatedCriterium}
The projective structure $\Pi$ defined by (\ref{eq:ProjStructDiffEqBis})
is foliated if, and only if,  there exist functions $f,g\in\OO(\Omega)$ such that 

\begin{equation}\label{eq:IntegEqfg}
\left\{\begin{matrix} 
f_x&=&f^2-Bf+Ag+B_x-A_y\\
g_x-f_y&=&2\left(fg-Bg-Cf+BC-AD\right)\\
g_y&=&-g^2+Cg-Df+C_y-D_x
\end{matrix}\right.
\end{equation}
\end{prop}

We want to understand if the existence of a foliated structure at the neighborhood of a point of $\Omega$
propagates on the whole of $\Omega$, maybe assuming $\Omega$ simply connected. 
We first notice that $\Omega$ might have meromorphic extension even if the projective structure
is holomorphic as shown by the following example:

\begin{example}\label{ex:FlatMeroForHoloProj}
Consider the linear projective structure $\Pi_0$ defined by $y''=0$. Then the meromorphic Riccati equation
$$dz=\frac{dy-z dx}{x}$$
is Frobenius integrable and induces the projective structure $\Pi_0$. In fact, it defines a pencil of foliations
by lines, namely the family of radial foliations with center along $x=0$. The foliated structure has a pole
along $x=0$.
\end{example}

Therefore, we cannot expect to have a holomorphic prolongation in general, but we are going to prove that 
it admits a meromorphic prolongation to $\Omega$ assuming that $\Omega$ is simply connected.
Let us first prove it assuming that the projective structure is normalized with $A=D=0$.

\begin{lemma}\label{Lem:ProlongationFolAD0}
Assume that $\Omega$ is a polydisc and the projective structure $\Pi$ is given 
by (\ref{eq:ProjStructDiffEqBis}) with $A=D=0$ on $\Omega$.
Then any local $f,g$ holomorphic at a point $p\in\Omega$ and satisfying (\ref{eq:IntegEqfg}) 
extend meromorphically on the whole of $\Omega$.
\end{lemma}

Recall that $A=D=0$ is equivalent to say that the foliations $dx=0$ and $dy=0$ are $\Pi$-geodesic.

\begin{proof}Assume that $f$ and $g$ are holomorphic on an open ball $U\subset\Omega=\Disc_1\times \Disc_2$
satisfying equations (\ref{eq:IntegEqfg}). We want to prove that we can extend
$f$ and $g$ meromorphically on $\Omega$; indeed, in that case, the Riccati equation
(\ref{eq:FromProjToRicc}) will be integrable on $U$, and therefore on the whole of $\Omega$
by analytic continuation. 
Setting $A=D=0$ in (\ref{eq:IntegEqfg}) gives
\begin{equation}\label{eq:IntegEqfgCD0}
\left\{\begin{matrix} 
f_x&=&f^2-Bf+B_x\\
g_x-f_y&=&2\left(fg-Bg-Cf+BC\right)\\
g_y&=&-g^2+Cg+C_y
\end{matrix}\right.
\end{equation}
The basic idea of the proof is to use successively the different differential equations 
(all of Riccati type either in $f$ or $g$) to extend the solutions $f,g$ firstly along a vertical cylinder,
and then horizontally on the whole of $\Omega$. However, we have to be careful because
solutions may have poles, and equations also after substitution.

Denote by $\pi_1(x,y)=x$ and $\pi_2(x,y)=y$ are the projections on coordinates,
and consider the vertical cylinder $U_1:=\pi_1^{-1}(\pi_1(U))$.
The last equation of (\ref{eq:IntegEqfgCD0}) is a Riccati equation in $g$ over $U_1$ with respect to variable $y$.
As we already get a solution on $U$, the Riccati equation provides by integration w.r.t. $y$ a meromorphic
extension along the vertical cylinder $U_1$, solving the same equation, that we still denote by $g$. 
Substituting the meromorphic solution $g$ in the second equation yields a meromorphic Bernoulli equation 
on $U_1$ for $f$ with respect to variable $y$. 
Therefore, the holomorphic solution $f$ on $U$ admits an analytic continuation along any path in $U_1\setminus(g)_\infty$,
where $(g)_\infty$ denotes the polar locus of $g$ in $U_1$. However, $f$ needs not be meromorphic at $(g)_\infty$
{\it a priori}, and might have multiform extension to the complement. 

We claim that $f$ has meromorphic extension along non horizontal branches of $(g)_\infty$. Indeed, if $\Gamma$ is such 
a branch, then the restriction of $f$ to a general line $L:\{y=\text{constant}\}$ intersecting $\Gamma$ also intersects
$U_1\setminus (g)_\infty$ where $f$ admits an analytic continuation, and the first equation shows that $f$ extends
meromorphically all along $L$, and in particular across $\Gamma$. Therefore, $f$ admits a meromorphic continuation 
along any path inside $U_1^*=U_1\setminus\Gamma_2$ where $\Gamma_2$ is the union of horizontal branches of $(g)_\infty$,
i.e. fibers of $\pi_2$;
moreover, the polar locus $(f)_\infty$ has no horizontal branch (apart from $\Gamma_2$ where it is not defined). 

We now consider the first equation which allows us to provide a meromorphic continuation of $f$ along all horizontal lines
$L$ intersecting $U_1^*$. In other words, $f$ admits a meromorphic continuation along any path inside 
$V_1^*=\pi_2^{-1}\pi_2(U_1^*)$; if we still denote by $\Gamma_2$ the extension of horizontal branches 
of $(g)_\infty$ in $\Omega$, then we see that $V_1=\Omega\setminus \Gamma_2$.

Finally, we use the second equation again to provide the analytic continuation of $g$ in the horizontal paths.
Since we have to substitute $f$ in that equation, we get, as before, that $g$ admits meromorphic continuation
along $V_2=\Omega\setminus\Gamma_1$ where $\Gamma_1$ is contained into vertical branches of the polar locus $(f)_\infty$.

So far, we have proved that the foliated structure can be continuated meromorphically 
along every path avoiding the union $\Gamma_1\cup\Gamma_2$ of vertical and horizontal geodesics.
We have that $f$ and $g$ have a pole at $\Gamma_1$ and $\Gamma_2$ respectively,
but $g$ and $f$ do not {\it a priori} extend meromorphically on $\Gamma_1$ and $\Gamma_2$ respectively,
and might have monodromy around.
In order to conclude the proof, we have to check that these are fake obstructions and the foliated structure 
extends meromorphically on the whole of $\Omega$. One way to prove this is to notice that 
the vertical and horizontal geodesic foliations have been choosen arbitrarily: we can repeat 
our arguments in other coordinates, by normalizing two geodesic foliations that are transversal to 
both $\Gamma_1$ and $\Gamma_2$, so that these curve become neither horizontal, nor vertical
in the new coordinates. Another way to prove it is to make a closer analysis at differential equations 
near a generic point of $\Gamma_1$, say. The first equation can be rewritten as
$$\frac{f_x-B_x}{f-B}=f$$
so that we see that the meromorphic solution $f$ must have simple poles with constant residue $-1$. Now, let $x=0$ such a pole.
The second equation is a Bernoulli equation for $g$ with simple poles, of the form
$$g_x=\left(-\frac{2}{x}+\text{holomorphic}\right)\cdot g+\frac{\text{holomorphic}}{x};$$
but such equation has only meromorphic solutions: indeed, $x^2g$ satisfies a holomorphic Bernoulli equation 
and is therefore holomorphic. We can make a similar discussion with horizontal branches $\Gamma_2$.
\end{proof}

\begin{cor}\label{lem:ProlongationFol}
Assume that $\Omega\subset\CC^2$ is an open set equipped with a projective structure $\Pi$.
If $\Pi$ is foliated at the neighborhood of some point $p\in\Omega$, then the foliated structure
(or Riccati foliation) can be meromorphically continuated along any path in $\Omega$.
\end{cor}

\begin{proof}It suffices to cover the path by local polydiscs on which we can normalize $\Pi$
as in Proposition \ref{Lem:ProlongationFolAD0}, and then use that proposition successively 
along the path to propagate the Riccati structure.
\end{proof}

A similar prolongation result has been obtained for first integrals of the affine connection
in \cite{ContattoDunajski2} (degree 1) and \cite{BDE} (degree 2). However, these does not 
imply similar results for first integrals of the projective structure. And indeed, in the case 
of affine connections, there are no pole along prolongation.

\section{Classification of neighborhoods of rational curves}\label{sec:ClassifMishNeigh}
Let $\mathbb P^1\hookrightarrow S$ be an embedding of $\mathbb P^1$
into a smooth complex surface and let $C$ be its image. The self-intersection 
of $C$ is also the degree of the normal bundle of the curve  $C\cdot C=\deg(N_C)$.
When $C\cdot C< 0$, it follows from famous work of Grauert \cite{Grauert} 
that the germ of neighborhood $(S,C)$ is linearizable, i.e. biholomorphically
equivalent to $(N_C,0)$ where $0$ denotes the zero section in the total space of the normal bundle. 
Such neighborhood 
is called rigid since there is no non trivial deformation. When $C\cdot C=k\ge 0$,
it follows from Kodaira \cite{Kodaira} that the deformation space of the curve $C$
in its neighborhood is smooth of dimension $k+1$. In particular, for $C\cdot C=0$,
the neighborhood is a fibration by rational curves, which is thus trivial by Fisher-Grauert
\cite{FischerGrauert}: the neighborhood is again linearizable (see also Savel'ev \cite{Savelev}), thus rigid.
However, in the positive case $C\cdot C>0$, it is also well-known that we have huge moduli.
The analytic classification is due to Mishustin \cite{Mishustin} 
but a formal version was already given by Hurtubise and Kamran in \cite{Hurtubise} one year before;
in this section, we recall the case $C\cdot C=1$.

Let us first decompose $C=V_0\cup V_\infty$ where $x_i:V_i\stackrel{\sim}{\to}\mathbb C$
are affine charts, $i=0,\infty$, with $x_0x_\infty=1$ on $V_0\cap V_\infty$.
Then any germ of neighborhood $(S,C)$ can be decomposed as the union 
$U_0 \cup U_\infty$ of two trivial neighborhoods $U_i \simeq V_i \times \mathbb{D}_{\epsilon}$ with coordinates $(x_i,y_i)$ patched together by a holomorphic map
$$
(x_{\infty}, y_{\infty})=\Phi(x_0, y_0)=\left( \frac{1}{x_0} + \sum_{n\geq 1} a_{n}(x_0) y_{0}^{n}\ ,\ \sum_{n\geq 1} b_{n}(x_0) y_{0}^{n} \right)
$$
where $a_n,b_n$ are holomorphic on $V_0\cap V_\infty\simeq\mathbb C^*$. Moreover, $b_1$
does not vanish on $V_0\cap V_\infty$ and, viewed as a cocycle $\{b_1\} \in H^1(\Pp^1, \OO_C^{*})$,
defines the normal bundle $N_C$. Denote $U_\Phi$ the germ of neighborhood defined by such a gluing map. The gluing map $\Phi$ can also be viewed as a non linear cocycle encoding the biholomorphic class 
of the neighborhood, as illustrated by the following straightforward statement.

\begin{prop}\label{prop:EquivalentCocycles}
Given another map $\Phi'$, then the following data are equivalent:
\begin{itemize}
\item a germ of biholomorphism $\Psi:U_\Phi\stackrel{\sim}{\to} U_{\Phi'}$ inducing the identity on $C$,
\item a pair of biholomorphism germs 
$$
\Psi^{i}(x_i, y_i) =\left( x_i + \sum_{n\geq 1} a^{i}_{n}(x_i) y_{i}^{n}\ ,\ \sum_{n\geq 1} b^{i}_{n}(x_i) y_{i}^{n} \right), \hspace{0.3cm} (i=0, \infty)
$$
(with $b_1^0,b_1^\infty$ not vanishing) satisfying $\Phi'\circ\Psi^0=\Psi^{\infty}\circ\Phi$:
\end{itemize}
\xymatrix{
&&&&&U_0 \ar[r]^{\Phi} \ar[d]_{\Psi^0} & U_{\infty}\ar[d]^{\Psi^{\infty}}  \\
&&&&&U_0 \ar[r]^{\Phi'}  & U_{\infty}  
}
We will say that the two ``cocycles'' $\Phi$ and $\Phi'$ are equivalent in this case.
\end{prop}

Since $H^1(\Pp^1, \OO_C^*) = \mathbb{Z}$ there exist $b^i \in \OO^*(V_i)$, $i=0,
\infty$, such that $b^{\infty} b_1  = x_0^k b^0$. Thus, the pair $\Psi^i(x_i,y_i)=(x_i,b^iy_i)$,
$i=0,\infty$, provides us with an equivalent cocycle such that $b_1(x_0)=x_0^k$. 
Now, this exactly means that $C\cdot C =-k$. As conclusion, $(+1)$-neighborhoods
can be defined by a cocycle of the form
$$
\Phi(x_0, y_0)=\left( \frac{1}{x_0} + \sum_{n\geq 1} a_{n}(x_0) y_{0}^{n}\ ,\ \frac{y_0}{x_0} + \sum_{n\geq 2} b_{n}(x_0) y_{0}^{n} \right)=:\left(\frac{1}{x_0} +a\ ,\ \frac{y_0}{x_0}+b\right).
$$

\subsection{Normal form}
Using the equivalence defined in Proposition \ref{prop:EquivalentCocycles} above,
we can reduce the cocycle $\Phi$ into an almost unique normal form:

\begin{thm}[Mishustin]\label{Mishustin}
Any germ $(S,C)$ of $(+1)$-neighborhood is biholomorphically equivalent to a germ $U_\Phi$
for a cocycle $\Phi$ of the following ``normal form''
\begin{equation}\label{Eq:NorForMishustin}
\Phi = \left( \frac{1}{x} + \sum_{V(-3,4)} a_{m,n}x^m y^n\   ,\  \frac{y}{x} + \sum_{V(-2,3)} b_{m,n} x^m y^n\right),
\end{equation}
where $V(k,l):=\{(m+k,n+l)\in\mathbb Z^2\ ;\ -n\le m\le 0\},\ \ \ (k,l)\in\mathbb Z^2$.

Moreover, when the neighborhood $(S,C)$ admits a fibration transverse to $C$, then one can choose 
all $a_{m,n}=0$ so that the fibration is given by $x_0=\frac{1}{x_\infty}:S\to C$.
\end{thm}

As we shall see in the next section, this normal form is unique up to a $4$-dimensional group action.

\begin{remark}
This normal form was firstly established by Hurtubise and Kamran in \cite{Hurtubise},
but without convergence. They obtained it by an iterative procedure, where the $y^n$-part of the coefficients
are simplified at the $n^{\text{th}}$ step. In fact, when we modify the cocycle as in Proposition \ref{prop:EquivalentCocycles}, $\Phi'=\Psi_\infty\circ\Phi\circ\Psi_0^{-1}$,
it can be shown that:
\begin{itemize}
\item[] $\Psi_0$ (resp. $\Psi_\infty$) can be choosen at each step so as to eliminate all coefficients $a_{m,n},b_{m,n}$ with $(m,n)$ lying 
in the corresponding pink (resp. yellow)  area in Figure \ref{Fig:Normal Form}.

\end{itemize}
Mishustin's proof in \cite{Mishustin} is of very different nature though. 
Denote $p_0$ and $p_\infty$ the points $x_0=0$ and $x_\infty=0$ in $C$. After blowing-up
$p_\infty$, the neighborhood of the strict transform of $C$ becomes a product $C\times(\CC,0)$
with global coordinates $(x_0,y_0)$; these define coordinates on some neighborhood $U_0$ of $V_0$ in $U$. By blowing-up $p_0$ instead of $p_\infty$, we get coordinates $(x_\infty,y_\infty)$ on a neighborhood of $U_\infty$ of $V_\infty$. It turns out that this choice of coordinates provide the normal form (in a geometric, and therefore convergent way)
up to a last normalization. Full details are given in \cite{MaycolFrank}.
\end{remark}

\begin{figure}[ht!]
\includegraphics[scale=0.4]{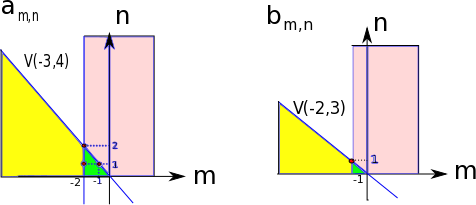}
\caption{Normal Form}
\label{Fig:Normal Form}
\end{figure}

\subsection{Isotropy group for normal forms}
During the proof of Theorem \ref{Mishustin}, we have the possibility to normalize coefficients 
$$a_{-1,1}=a_{-2,1}=a_{-2,2}=0\ \ \ \text{and}\ \ \ b_{-1,1}=1$$
by either using $\Psi^0$ or $\Psi^\infty$ (the green intersecting area of Figure \ref{Fig:Normal Form}). This underline a $4$-parameter degree of freedom in the choice
of normalizing coordinates systems $(x_0,y_0)$ and $(x_\infty,y_\infty)$.

For instance, if $\Phi_0=(\frac{1}{x},\frac{y}{x})$ is the linear neighborhood, then we know that it admits 
the following family of automorphisms:
$$\left(\frac{x_\infty+\alpha y_\infty}{1+\beta y_\infty},\frac{\theta y_\infty}{1+\beta y_\infty}\right)\circ\Phi_0=
\Phi_0\circ\left(\frac{x_0+\beta y_0}{1+\alpha y_0},\frac{\theta y_0}{1+\alpha y_0}\right).$$
We will see that this group acts on the set of normal forms (\ref{Eq:NorForMishustin}), 
having $\Phi_0$ as fixed point.
On the other hand, we can easily check that if $\Phi$ is in normal form, then 
$$(x_\infty+\gamma y_\infty^2,y_\infty)\circ\Phi=\Phi'\circ(x_0-\gamma y_0^2,y_0)$$
gives another a new normal form. The $4$-parameter of freedom is a combination 
of those two actions. 

\begin{prop}\label{PropMishustinFreedom}
Consider a cocycle in normal form
$$
\Phi = \left( \frac{1}{x} + \sum_{V(-3,4)} a_{m,n}x^m y^n\   ,\  \frac{y}{x} + \sum_{V(-2,3)} b_{m,n} x^m y^n\right).
$$
Then an equivalent cocycle 
$\Psi^\infty\circ\Phi=\Phi'\circ\Psi^0$
is also in normal form if, and only if, there are constants $\alpha,\beta,\gamma\in\mathbb C$ 
and $\theta\in\mathbb C^*$
such that
\begin{equation}\label{eq:FreedomMishustin}\begin{matrix}
\Psi^\infty&=&\left(\frac{x_\infty+\alpha y_\infty}{1+\beta y_\infty}+\gamma\left(\frac{y_\infty}{1+\beta y_\infty}\right)^2+\frac{\beta k^\infty(y_\infty)}{(1+\beta y_\infty)^2}\ ,\ \frac{\theta y_\infty}{1+\beta y_\infty}\right)\\
\Psi^0&=&\left(\frac{x_0+\beta y_0}{1+\alpha y_0}-\gamma\left(\frac{y_0}{1+\alpha y_0}\right)^2-\frac{\alpha k^0(y_0)}{(1+ \alpha y_0)^2}
\ ,\ 
\frac{\theta y_0}{1+\alpha y_0}\right)
\end{matrix}\end{equation}
where $k^0(y_0)=\sum_{n\ge3}b_{-2,n}y_0^n$ and $k^\infty(y_\infty)=\sum_{n\ge3}b_{-n+1,n}y_\infty^n$.
\end{prop}

\begin{proof}
See \cite{MaycolFrank}, page $20$ for a detailed proof.
\end{proof}
For instance, starting with the linear neighborhood $\Phi_0=(\frac{1}{x},\frac{y}{x})$, then we obtain
the following equivalent cocycles in normal form (with $c=\frac{\gamma}{\theta^2}\in\mathbb C$ arbitrary)
$$\Phi=\left( \frac{1}{x}\frac{(1+2c\frac{y^2}{x})}{(1+c\frac{y^2}{x})^2} \ ,\  \frac{y}{x}\frac{1}{(1+c\frac{y^2}{x})}\right)\ \ \ \sim\ \ \ \Phi_0.$$

We promptly deduce from Proposition \ref{PropMishustinFreedom} that any change of normalization
$$\Phi'=\Psi^\infty\circ\Phi\circ(\Psi^0)^{-1}$$
of a given cocycle in normal form $\Phi$ is determined by the quadratic part of $\Psi^0$:
\begin{equation}\label{QuadPartNorm}
\Psi^0=(x+(\beta-\alpha x)y+(\alpha^2 x-(\alpha\beta+\gamma))y^2+\cdots\ ,\ 
\theta y-\theta \alpha y^2+\cdots).
\end{equation}
Conversely, for any $\vartheta=(\alpha,\beta,\gamma,\theta)\in\mathbb C^3\times\mathbb C^*$,
the above quadratic part can be extended as a new normalization $(\Psi^0_{\vartheta,\Phi},\Psi^{\infty}_{\vartheta,\Phi})$ for each cocycle $\Phi$ in normal form. 
We thus get an action of $\mathbb C^3\times\mathbb C^*$ on the set of normal forms
$$(\vartheta,\Phi)\ \mapsto\ \vartheta\cdot\Phi:=\Psi^\infty_{\vartheta,\Phi}\circ\Phi\circ(\Psi^0_{\vartheta,\Phi})^{-1}$$
with the group law given by 
$$\begin{matrix}
\vartheta_1\cdot(\vartheta_2\cdot\Phi)&=&\Psi^\infty_{\vartheta_1,\Phi'}\circ\left(\underbrace{\Psi^\infty_{\vartheta_2,\Phi}\circ\Phi\circ(\Psi^0_{\vartheta_2,\Phi})^{-1}}_{\Phi'}\right)\circ(\Psi^0_{\vartheta_1,\Phi'})^{-1}&&\\
&=&\left(\Psi^\infty_{\vartheta_1,\Phi'}\circ\Psi^\infty_{\vartheta_2,\Phi}\right)\circ\Phi\circ\left(\Psi^0_{\vartheta_1,\Phi'}\circ\Psi^0_{\vartheta_2,\Phi}\right)^{-1}&&\\
&=&\Psi^\infty_{\vartheta_3,\Phi}\circ\Phi\circ(\Psi^0_{\vartheta_3,\Phi})^{-1}&=&\vartheta_3\cdot\Phi
\end{matrix}$$
This group law can be easily computed by composing the quadratic parts (\ref{QuadPartNorm}) of 
$\Psi^0_{\vartheta_1}$ and $\Psi^0_{\vartheta_2}$, and we get
$$\vartheta_3=(\alpha_2+\theta_2\alpha_1,\beta_2+\theta_2\beta_1,\gamma_2+\theta_2^2\gamma_1,\theta_1\theta_2).$$
In other words, the group law on parameters $\vartheta=(\alpha,\beta,\gamma,\theta)$ is equivalent
to the matrix group law
$$\Gamma:=\left\{\begin{pmatrix}1&\alpha&\beta&\gamma\\ 0&\theta&0&0\\
0&0&\theta&0\\ 0&0&0&\theta^2\end{pmatrix}\ ,\ (\alpha,\beta,\gamma,\theta)\in\mathbb C^3\times\mathbb C^*\right\}\ \ \ \subset\ \mathrm{GL}_4(\mathbb C).$$
We deduce:

\begin{cor}\label{Cor:ActionLinGroup}
The $4$-dimensional matrix group $\Gamma$ acts on the set of normal forms (\ref{Eq:NorForMishustin})
as defined in Proposition \ref{PropMishustinFreedom} and the set of equivalence classes is in one-to-one
correspondance with the set of isomorphisms classes of germs of $(+1)$-neighborhoods $(S,C)$ of 
the rational curve $C\simeq\mathbb P^1_x$ (with fixed coordinate $x$).
\end{cor}

Let us describe this action on the first coefficients of the cocycle $\Phi'=\vartheta\cdot\Phi$:
\begin{equation}\label{eq:ActionCoeffCocycle}
\left\{\begin{matrix}
a_{-3,4}'=&\frac{a_{-3,4}-\gamma^2+2\gamma b_{-2,3}-\beta b_{-2,4}+\alpha b_{-3,4} }{\theta^4}\hfill\\
a_{-3,5}'=&\frac{a_{-3,5}+\alpha a_{-3,4}+(2\gamma-\alpha\beta) b_{-2,4}+\alpha^2 b_{-3,4} -\beta b_{-2,5} +\alpha b_{-3,5}}{\theta^5}\hfill\\
a_{-4,5}'=&\frac{a_{-4,5}+2\beta a_{-3,4} +3\beta b_{-2,3}^2 -\beta^2 b_{-2,4}+(\alpha\beta+2\gamma) b_{-3,4} -\beta b_{-3,5} +\alpha b_{-4,5}}{\theta^5}\\
&\cdots\hfill
\end{matrix}\right.
\end{equation}
$$\left\{\begin{matrix}
b_{-2,3}'=\frac{b_{-2,3}-\gamma}{\theta^2}\hfill & & \\
b_{-2,4}'=\frac{b_{-2,4}}{\theta^3}\hfill& b_{-3,4}'=\frac{b_{-3,4}}{\theta^3}\hfill &\\
b_{-2,5}'=\frac{b_{-2,5}+\alpha b_{-2,4}}{\theta^4}& b_{-3,5}'=\frac{b_{-3,5}+\alpha b_{-3,4}-2\gamma b_{-2,3}+\gamma^2}{\theta^4}\hfill & b_{-4,5}'=\frac{b_{-4,5}+\beta b_{-3,4}}{\theta^4}\\
\cdots&
\end{matrix}\right.$$

\subsection{Existence of transverse fibration}

We go back to the notion of transversal fibration by discs on $(S,C)$ considered in \ref{sec:TransvFibration}.
If we have such a fibration, it can be defined by a submersion $H:S\to C$ inducing the identity on $C$;
equivalently, after composition with $x:C\to\mathbb P^1$ we get an extension of the coordinate $x$
on the neighborhood $S$. In this case, recall (see Theorem \ref{Mishustin}) that one can choose 
a normal form with zero $a$-part:
$$\Phi(x,y)=\left( \frac{1}{x}   ,  \frac{y}{x} + \sum_{V(-2,3)}b_{m,n}x^my^n\right)$$
compatible with the fibration in the sense that $x\circ H=x_0=\frac{1}{x_\infty}$. 

\begin{prop}\label{prop:NormalFormFibration}
A $(+1)$-neighborhood with normal form $\Phi$ admits a transversal fibration if, and only if,
there is a $\vartheta=(\alpha,\beta,\gamma,1)\in\Gamma$ such that the $a$-part of $\vartheta\cdot\Phi$
is trivial. Moreover, the set of transversal fibrations is in one-to-one correspondance with the set
of those $(\alpha,\beta,\gamma)\in\mathbb C^3$ for which the $a$-part of $(\alpha,\beta,\gamma,1).\Phi$
is zero.
\end{prop}

\begin{proof} Just observe that, once we get a normal form with trivial $a$-part, 
the fibration given by $x_0=1/x_\infty$ is only preserved by the action of 
$(0,0,0,\theta)$. We thus have to divide the group action by this normal subgroup
to get a bijection with the set of fibrations.
\end{proof}

It is clear that we cannot kill the $a$-part in general by means of the above $3$-dimensional group action
and therefore that we have infinitely many obstructions to have a transversal fibration.

\begin{prop}\label{prop:4NeighbFibr}
Any $(+1)$-neighborhood admits a normal form with $a$-part vanishing up to the order $4$ in $y$-variable, 
i.e. with $a_{-3,4}=0$.
In other words, the $4^{\text{th}}$ infinitesimal neighborhood always admit a transverse fibration; the first obstruction 
to extend it arrives at order $5$.
\end{prop}

\begin{example}\label{ex:WithoutFibration}
The neighborhood $U_\Phi$ given by the cocycle in normal form
$$
\Phi = \left( \frac{1}{x}+ \frac{y^5}{x^3} +\sum_{V(-3,4)\cap \{n>5\}}a_{m,n}x^my^n, \frac{y}{x}  + \sum_{V(-2,3)\cap \{n>5\}}b_{m,n}x^my^n \right)
$$
does not admit transversal fibration. 
\end{example}

\begin{proof}[Proof of Proposition \ref{prop:4NeighbFibr} and Example \ref{ex:WithoutFibration}]
Looking back at the explicit action (\ref{eq:ActionCoeffCocycle}) of $\Gamma$ on the $a$-coefficients,
we see that whatever is $a_{-3,4}$, we can assume $a_{-3,4}'=0$ by setting $\alpha=\beta=0$, $\theta=1$
and  $\gamma^2=a_{-3,4}$. On the other hand, the coefficient $a_{-3,5}$ cannot be killed in general,
and in particular in the example. Indeed, since all other coefficients $a_{m,n},b_{m,n}$ occuring in $a_{-3,5}'$ - formula (\ref{eq:ActionCoeffCocycle}) - are zero,
we see that 
$a_{-3,5}'=\theta^{-5}\not=0$ whatever are $(\alpha,\beta,\gamma)\in\mathbb C^3$.
\end{proof}

\begin{example}\label{ex:WithOneFibration}
The neighborhood $U_\Phi$ given by the cocycle in normal form
$$
\Phi = \left( \frac{1}{x}, \frac{y}{x}  + \sum_{V(-2,3)\cap \{n\ge5\}}b_{m,n}x^my^n \right)\ \ \ \text{with}\ \ \ b_{-2,5}b_{-4,5}\not=b_{-3,5}
$$
admits no other transversal fibration than $dx=0$.
Indeed, following Proposition \ref{prop:NormalFormFibration}, another fibration would correspond
to a triple $(\alpha,\beta,\gamma)\not=(0,0,0)$ such that the $a$-part of $(\alpha,\beta,\gamma,1).\Phi$
is zero. However, formula (\ref{eq:ActionCoeffCocycle}) gives
$$a_{-3,4}'=-\gamma^2,\ \ \ a_{-3,5}'=\alpha b_{-3,5}-\beta b_{-2,5}\ \ \ \text{and}\ \ \ 
a_{-4,5}'=\alpha b_{-4,5} -\beta b_{-3,5}$$
which shows that we must have $\alpha=\beta=\gamma=0$.
\end{example}

From previous examples, we understand that neighborhoods with exactly one transversal fibration 
have infinite dimension and codimension in the moduli of all $(+1)$-neighborhoods.

\section{Neighborhoods with several transverse fibrations}\label{sec:transvfibr}
In this section, we study $(+1)$-neighborhoods having several fibrations. 
The following was recently announced and partly proved in \cite{CRAS};
Paulo Sad and the first author gave another proof in \cite{MaycolPaulo}.

\begin{thm}[\cite{CRAS,MaycolPaulo}]\label{3-fibrations}
If a germ $(S,C)$ of $(+1)$-neighborhood admits at least $3$ distinct fibrations $\mathcal H$, $\mathcal H'$ and $\mathcal H''$ transversal to $C$, 
then $(S, C)$ is equivalent to $(\PP, L)$, where $L$ is a line in $\PP$. 
 \end{thm}

In this case, recall (see example \ref{conexion-trivial}) that there is a $2$-parameter family
of transverse fibrations (each of them is a pencil of lines through a point) and any two 
have a common fiber. These results are very similar to \cite[Theorem 1]{ContattoDunajski2};
however, a foliated structure (or Riccati foliation) needs not lift as a linear first integral 
for the affine connection, and conversally we need two independent linear first integrals
to produce a Riccati foliation, i.e. a first integral for the projective structure having degree 1 in $z$.
Although the results obtained look very similar, the problems seem to be independant so far.

Before proving Theorem \ref{3-fibrations} is full details in section \ref{sec:proof3fibrat}, we need first to classify
pairs of fibrations on $(+1)$-neighborhoods.

\subsection{Tangency between two fibrations}
Given two (possibly singular) distinct foliations $\mathcal H$ and $\mathcal H'$ on a complex surface $X$, define the tangency divisor $\tang(\mathcal H,\mathcal H')$ as follows. Locally, we can define the two foliations
respectively by $\omega=0$ and $\omega'=0$ for holomorphic $1$-forms $\omega,\omega'$ without zero 
(or with isolated zero in the singular case); then the divisor  $\tang(\mathcal H,\mathcal H')$ is locally 
defined as the (zero) divisor of $\omega\wedge\omega'$.

\begin{prop}\label{prop:tang+1}
If a germ $(S,C)$ of $(+1)$-neighborhood admits $2$ distinct fibrations $\mathcal H$ and $\mathcal H'$,
then 
\begin{itemize}
\item either  $\tang(\mathcal H,\mathcal H')=[C]$ (without multiplicity),
\item or $\tang(\mathcal H,\mathcal H')\cdot [C]$ is a single point (without multiplicity).
\end{itemize}
\end{prop}

The former case is rigid and will be described in section \ref{tancency-along-C_0}.
In the latter case, the tangency divisor is reduced and transversal to $C$ (equivalently 
the restriction $\tang(\mathcal H,\mathcal H')\vert_{C}$ has degree one);
moreover, the support $\tang(\mathcal H,\mathcal H')$ of the divisor is
\begin{itemize}
\item either a common fiber of  $\mathcal H$ and $\mathcal H'$,
\item or is generically transversal to $\mathcal H$ and $\mathcal H'$ (but might be tangent at some point).
\end{itemize}

\begin{proof}
See \cite[Section 2]{MaycolPaulo} or \cite{MaycolFrank}.

\end{proof}

\subsection{Two fibrations having a common leaf}
Let us start with the simplest case.
\begin{thm}\label{hoja-comun}
Let $(S,C)$ be a $(+1)$-neighborhood that admits two transversal fibrations $\HH$ and $\HH'$ with a common leaf $T$.
Then $(S,C)$ is linearizable. 
\end{thm}

\begin{proof}
By Proposition \ref{prop:tang+1}, $\HH$ and $\HH'$ have contact of order $1$ along $T$ and 
are transversal outside. Let $H,H':S\to\mathbb P^1$ be the two submersions defining these foliations 
and coinciding with $x:C\to \Pp^1$ in restriction to $C$. For simplicity, assume $T=\{H=\infty\}=\{H'=\infty\}$.
We can use $H$ and $H'$ as a system of coordinates to embed $(S,C)$ into $\mathbb P^2$
as illustrated in Figure \ref{fig:ejemplo}. Precisely, 
consider the map given in homogeneous coordinate $(u:v:w)\in\mathbb P^2$ by
$$\Phi:=(H:H':1):S\to\mathbb P^2.$$
\begin{figure}[ht!]
  \centering
    \includegraphics[scale=0.5]{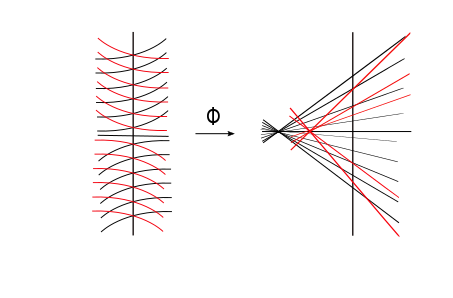}
  \caption{Linearization}
  \label{fig:ejemplo}
\end{figure}
The complement $S\setminus T$ is clearly embedded by $\Phi$ as a neighborhood of the diagonal 
$\Delta=\{u=v\}$ in the chart $w=1$. Moreover, the two fibrations are send to fibrations $du=0$ and $dv=0$.
We just have to check that this map is (well-defined and) still a local diffeomorphism at $T\cap C$.
In local convenient coordinates $(x_\infty,y_\infty)\sim(0,0)$ at the source, we have 
$$\frac{1}{H}=x_\infty\ \ \ \text{and}\ \ \ \frac{1}{H'}=x_\infty(1+y_\infty\cdot f(x_\infty,y_\infty)),\ \ \ f(0,0)\not=0$$
(we used that $1/H$ and $1/H'$ coincide along $y_\infty=0$, vanish at $x_\infty=0$ and have simple tangency there).
Coordinates at the target are given by 
$$(X,Y)=\left(\frac{1}{u},\frac{u}{v}-1\right).$$
Therefore, our map is given by
$$\Phi\ :\ (X,Y)=(x_\infty,y_\infty\cdot f(x_\infty,y_\infty))$$
which is clearly a local diffeomorphism.
\end{proof}

\subsection{Two fibrations that are tangent along a rational curve}\label{tancency-along-C_0}
The goal of this section is to describe a very special neighborhood. As we shall see later,
it is the only one having a large group of symmetries (i.e. of dimension $>2$) but not linearizable. 

Let us consider the diagonal $\Delta\subset\mathbb P^1\times\mathbb P^1$. The self-intersection 
is $\Delta\cdot\Delta=2$ and its neighborhood is a $(+2)$-neighborhood. Therefore, we can take 
(see section \ref{sec:GeralSelfInt}) the $2$-fold ramified cover, ramifying over $\Delta$:
$$\pi:(S,C)\stackrel{2:1}{\rightarrow} (\mathbb P^1\times\mathbb P^1,\Delta)$$
and we get a $(+1)$-neighborhood. Moreover, the two fibrations 
on $\mathbb P^1\times\mathbb P^1$ defined by projections on the two factors lift 
as fibrations $\HH_1$ and $\HH_2$ on $S$ transversal to $C$ whose tangent locus is
$\tang(\HH_1, \HH_2)=C$.  

\begin{remark}The two fibers passing through a given point $p\in S$ close enough to $C$
intersect twice: the Galois involution $\iota:(S,C)\to (S,C)$ permutes these two points.
\end{remark}

\begin{prop} The germ $(S,C)$ is not linearizable.
\end{prop}
\begin{proof}Assume by contradiction that $(S,C)$ is equivalent to the neighborhood of a line $L\subset\PP$.
Then each foliation $\HH_i$ extends as a global singular foliation on $\PP$. 
Since $\HH_i$ is totally transversal to $L$, it must be a foliation of degree $0$, i.e. a pencil of lines. But if $\HH_1$
and $\HH_2$ are pencil of lines, their tangency must be invariant (the line through the to base points), contradiction.
\end{proof}

\begin{cor}
The germ $(S,C)$ is not algebrizable, but the field $\mathcal M(S,C)$ of meromorphic functions
has transcendance degree $2$ over $\CC$.
\end{cor}

\begin{proof}
It immediately follows from Proposition \ref{prop:ProjectiveSurface} that $(S,C)$ is not algebrizable;
the field $\mathcal M(S,C)$ contains the field or rational functions on $\Pp^1\times\Pp^1$
which has indeed, as an algebraic surface, transcendance degree $2$ over $\CC$.
\end{proof}
\begin{remark}The fundamental group $\pi_1(\Pp^1\times\Pp^1\setminus\Delta)$ is trivial, 
and this is a reason why we cannot extend the ramified cover to the whole of the algebraic surface.
It is proved in \cite{MaycolFrank2} that there exist neighborhoods without non constant meromorphic functions.
\end{remark}

The cocycle $\Phi$ defining the germ $(S,C)$ can be constructed as follows. 
We can give local coordinates $(x_0, y_0)$, $(x_{\infty}, y_{\infty})$ on $S$ and $(u_0,v_0)$, $(u_\infty,v_\infty)$ on $V\supset\Delta$ such that
\begin{equation}\label{eq:CoordSpecialCovering}
\left\{\begin{matrix}
u_0&=&x_0\\
v_0&=&x_0 - y_0^2
\end{matrix} \right. ,
\hspace{0.5cm} \left\{\begin{matrix}
u_\infty&=&x_{\infty}\\
v_\infty&=&x_{\infty} + y_{\infty}^2
\end{matrix} \right. ,
\hspace{0.5cm} (u_\infty,v_\infty) = \left( \frac{1}{u_0}, \frac{1}{v_0}  \right)
\end{equation}
where $\Delta=\{v_0=u_0\}=\{v_\infty=u_\infty\}$. So the cocycle is explicitely given by 
\begin{equation}\label{eq:SpecialCocycle}
\Phi(x_0, y_0)=\left( \frac{1}{x_0}, \frac{y_0}{x_0} \left(  1- \frac{y_0^2}{x_0} \right)^{-1/2}   \right)= \left( \frac{1}{x_0}, \frac{y_0}{x_0} + \frac{y^3}{2x^2} + \frac{3y^5}{8x^3}+\ldots   \right),
\end{equation}
which is already in normal form. The fibrations are given by 
$$h_1=x_0=\frac{1}{x_\infty}\ \ \ \text{and}\ \ \ h_2=x_0 - y_0^2=\frac{1}{x_{\infty}+ y_{\infty}^2}.$$

\begin{prop}\label{prop:only2specialSL2}
The fibrations $\HH_1$ and $\HH_2$ are the only one fibrations on $S$ that are transverse to $C$.
\end{prop}
\begin{proof}
Recall (see Proposition \ref{prop:NormalFormFibration}) that transverse fibrations are in one-to-one correspondance with $\vartheta=(\alpha, \beta, \gamma, \theta) \in \mathbb{C}^3 \times \{1\}$ such that the $a$-part of the equivalent cocycle $\vartheta\cdot\Phi$ is zero. Substituting the explicit cocycle above in formula (\ref{eq:ActionCoeffCocycle}), we get
$$a_{-3,4}'=\gamma(1-\gamma),\ \ \ a_{-3,5}'=\frac{3}{8}\alpha\ \ \ \text{and}\ \ \ a_{-4,5}'=\frac{3}{8}\beta$$
so that the only two possibilities are $(\alpha, \beta, \gamma, \theta)=(0,0,0,1)$ or $(0,0,1,1)$ which respectively correspond
to the two fibrations $\HH_1$ and $\HH_2$.
\end{proof}

\begin{thm}\label{thm:RigiditySpecialExample}
Let $(S,C_0)$ be a $(+1)$-neighborhood and suppose that there are two fibrations $\HH_1$ and $\HH_2$ transverse to $C$ such that $\tang(\HH_1, \HH_2)=C_0$. Then $(S,C_0)$ is the previous example. 
\end{thm}

\begin{proof}
Let $x: C_0 \rightarrow \Pp^1$ be a global coordinate on $C_0$ and consider $h_1,h_2 : S \rightarrow \Pp^1$ 
first integrals of $\HH_1$ and $\HH_2$ such that $h_1|_{C_0} = h_2|_{C_0} =x$. Consider the map
$$
\pi :S \rightarrow  \Pp^1 \times \Pp^1\ ;\  \Phi(p) = (h_1(p), h_2(p)).
$$
Clearly, $\pi\vert_{C_0}: C_0  \stackrel{\simeq}{\rightarrow} \Delta$ is an isomorphism, and we claim that
$\pi$ is a $2$-fold covering of a neighborhood $V$ of $\Delta$, ramifying over $\Delta$.
In order to prove this claim, it suffices to check it near $p_0:\{x=0\}$ since $x$ is well-defined up to 
a Moebius transform. Fix local coordinates $(x,y)$ on $(S,p_0)$ such that $C_0=\{y=0\}$ 
and $h_1(x,y)=x$. Therefore, $h_2(x,y)=x-y^2f(x,y)$ with $f(0,0)\not=0$; here we have used 
that $h_1$ and $h_2$ coincide on $C_0$ and are tangent there, without multiplicity.
We can change coordinate $(X,Y)=(x,y\sqrt{f(x,y)})$ so that 
$$h_1(X,Y)=X\ \ \ \text{and}\ \ \ h_2(X,Y)=X-Y^2.$$
From this, it is clear that $\pi$ is a $2$-fold cover ramifying over $\Delta\subset\Pp^1_u\times\Pp^1_v$
since in coordinates $(U,V)=(u,u-v)$ we have $\Delta=\{V=0\}$ and $\pi:(X,Y)\mapsto(U,V)=(X,Y^2)$.
By construction, $\HH_1$ and $\HH_2$ are sent to $\ker(dX)$ and $\ker(dY)$.
\end{proof}

Now we want to write explicitly the differential equation associated to this example. 
In order to do that, we consider the automorphism group 
$$\mathrm{Aut}(\Pp^1_u\times\Pp^1_v,\Delta)=\mathrm{PSL}_2(\CC)\times \mathbb Z/2\mathbb Z$$
where the $\mathrm{PSL}_2(\CC)$-action is diagonal
$$\begin{pmatrix}a&b\\ c& d\end{pmatrix} \ :\ (u,v)\mapsto \left(\frac{au+b}{cu+d},\frac{av+b}{cv+d}\right)$$
and the $\mathbb Z/2\mathbb Z$-action is generated by the involution $(u,v)\mapsto(v,u)$. 
After $2$-fold cover $(S,C)\to(\Pp^1_u\times\Pp^1_v,\Delta)$, we get an action 
of 
$$\Gamma\simeq\{M\in\mathrm{GL}_2(\mathbb C)\ ;\ \det(M)=\pm1\}$$
where $-I$ is the Galois involution of the covering, and we have 
$$\mathrm{Aut}(\Pp^1_u\times\Pp^1_v,\Delta)=\Gamma/\{\pm I\}.$$
Indeed, the $\mathrm{PGL}_2(\mathbb C)$-action is given by $\mathrm{SL}_2(\mathbb C)/\{\pm I\}$ 
and $\langle\sqrt{-1}I\rangle/\{\pm I\}\simeq\mathbb Z/2\mathbb Z$ is the permutation of coordinates $u\leftrightarrow v$.
In coordinates $(u,v)=(x,x-y^2)$ (see (\ref{eq:CoordSpecialCovering})), the $\mathrm{SL}_2(\CC)$-action on $(S,C)$ writes
$$\begin{pmatrix}a&b\\ c& d\end{pmatrix} \ :\ 
(x,y)\mapsto\left(\frac{ax+b}{cx+d},\frac{y}{cx+d}\sqrt{\frac{cx+d}{cx+d-y^2}}\right)$$
where $ad-bc=1$ and the square-root choosen so as $\sqrt{1}=1$ (note that its argument is $1$ along $y=0$).
The involution writes
$$\begin{pmatrix}\sqrt{-1}&0\\ 0& \sqrt{-1}\end{pmatrix} \ :\   (x,y)\mapsto (x-y^2,\sqrt{-1}y).$$
Going to the dual picture, we get a projective structure $\Pi$ on $(\mathbb C^2,0)$ invariant by an action 
of the same group $\Gamma$, and fixing the origin $0$. By Bochner Linearization Theorem, 
the action of the maximal compact subgroup $\Gamma_0\subset\Gamma$ is holomorphically linearizable, 
and since $\Gamma$ is just the complexification of $\Gamma_0$ (and therefore Zariski dense in $\Gamma$)
the action of $\Gamma$ itself is linearizable.

\begin{prop}The unique non-trivial projective structure $\Pi$ on $(\mathbb C^2,0)$ which is invariant
by the linear action of $\mathrm{SL}_2(\CC)$ is given (up to homothety) by
\begin{equation}\label{eq:VerySpecialProjStr}
y'' = (x y' - y)^3.
\end{equation}
\end{prop}

\begin{proof} 
See \cite[Theorem 3]{Rom}.
\end{proof}

\begin{thm}
The two pencils of geodesic foliations for the $\mathrm{SL}_2(\CC)$-invariant projective structure $y'' = (x y' - y)^3$
are given by 
$$\omega_t^\pm=\left(y^2dx-(xy+\sqrt{-1})dy\right)+t\left((xy+\sqrt{-1})dx-x^2dy\right)=0,\ \ \ t\in\mathbb P^1$$
where $\pm$ stand for the two determinations of $\sqrt{-1}$. 
For each $t$, we see that the line $y=tx$ is a common leaf for the two foliations
$\mathcal F_{\omega_t^+}$ and $\mathcal F_{\omega_t^-}$.
\end{thm}

\begin{proof}Recall (see Proposition \ref{prop:only2specialSL2}) that there are exactly $2$ 
pencils of geodesic foliations for this special projective structure. 
We can just verify by computation that the pencils in the statement are geodesic, however
we think that it might be interesting to explain how we found them.

In order to find the elements of the pencils, we will use again the  $\mathrm{SL}_2(\CC)$-invariance. 
We first construct the two foliations (i.e. for the two pencils) having the geodesic $y=0$ as a special leaf,
i.e. $\omega_0^{\pm}$; then it will be easy to deduce the full pencil by making $\mathrm{SL}_2(\CC)$ acting on $\omega_0^{\pm}$.
In order to characterize $\omega_0^{\pm}$, let us go back to the dual picture $(S,C)$. The two foliations we are looking for 
come from the two fibers passing through $p_0=\{x_0=y_0=\infty\}$ in $C$, or equivalently the two fibers $u_0=\infty$ and $v_0=\infty$ 
on $\mathbb P^1\times\mathbb P^1$ (here we use coordinates given by formula (\ref{eq:CoordSpecialCovering})).
These curves are invariant by the action of the Borel subgroup 
$$\begin{pmatrix}a&b\\ 0& a^{-1}\end{pmatrix} \ :\ 
\left\{\begin{matrix}
(u_0,v_0)\mapsto \left(a^2u_0+ab,a^2v_0+ab\right)\hfill\\
(x_0,y_0)\mapsto\left(a^2x_0+ab,\frac{ay_0}{\sqrt{1-ay_0^2}}\right)
\end{matrix}\right\}$$
and they are the only one except the diagonal $\Delta$. Therefore, the two foliations we are looking for
are the only one that are invariant under the action of the Borel subgroup above, and distinct to the radial foliation.
Let $\omega=dy-f(x,y)dx$ be one of these foliations. Setting $(a,b)=(1,t)$, we get $\varphi^t(x,y)=(x+ty,y)$ and
$$(\varphi^t)^*\omega\wedge\omega=0,\ \forall t\ \ \ \Leftrightarrow\ \ \ yf_x+f^2=0\ \ \ 
\text{(where }f_x:=\frac{df}{dx}\text{)}.$$
Now, setting $(a,b)=(e^t,0)$, we get $\varphi^t(x,y)=(e^tx,e^{-t}y)$ and
$$(\varphi^t)^*\omega\wedge\omega=0,\ \forall t\ \ \ \Leftrightarrow\ \ \ 2f+xf_x-yf_y=0\ \ \ 
\text{(and }f_y:=\frac{df}{dy}\text{)}.$$
Finally, the foliation $\mathcal F_\omega$ is geodesic if, and only if, the corresponding surface $\{z=f\}$
in $\mathbb P(T\mathbb C^2)$ is invariant by the geodesic foliation defined by 
$v=\partial_x+z\partial_y+(xz-y)^3\partial_z$. In other words
$$ i_vd(z-f)\vert_{z=f}=0\ \ \ \Leftrightarrow\ \ \ (xf-y)^3-f_x-f f_y=0.$$
The combination of the three constraints gives (after eliminating $f_x$ and $f_y$)
$$\left(f-\frac{y}{x}\right)\left(f-\frac{y^2}{xy+\sqrt{-1}}\right)\left(f-\frac{y^2}{xy-\sqrt{-1}}\right)=0.$$
The first factor gives the radial foliation and the two other ones yield $\omega_0^{\pm}$.
We conclude by applying the one-parameter subgroup $\varphi^t(x,y)=(x,tx+y)$ to each of these foliations
and get the two pencils.\end{proof}

\begin{remark} It is interesting to complete the picture by the well-know two-fold cover 
$$\pi'\ :\ \left\{\begin{matrix}
\mathbb P^1\times\mathbb P^1&\to&\Pp^2\\
\Delta&\to&C
\end{matrix}\right.\ ;\ \ \ 
([u_0: u_{\infty}], [v_0: v_{\infty}])\mapsto [u_0 v_0 : u_0 v_{\infty}+ u_{\infty}v_0 : u_{\infty}v_{\infty}]$$
ramifying over $C:=\pi'(\Delta)$,  a conic. This map $\pi'$ is the quotient by the involution $\varphi:(u,v)\mapsto(v,u)$
on $\mathbb P^1\times\mathbb P^1$ that fixes the diagonal $\Delta$. We therefore have successive 
ramified covers
$$(S,C_0)\stackrel{\pi}{\to}(\mathbb P^1\times\mathbb P^1,\Delta)\stackrel{\pi'}{\to}(\Pp^2,C)$$
described in Figure \ref{fig:covering}. 
Since $(S,C_0)$ is not linearizable, the same holds for the two other ones (linearization can be lifted).
So the $(+4)$-neighborhood of a conic $C\subset \mathbb P^2$ is not linearizable. In fact, we can say more.
\begin{figure}[ht!]
  \centering
    \includegraphics[scale=0.50]{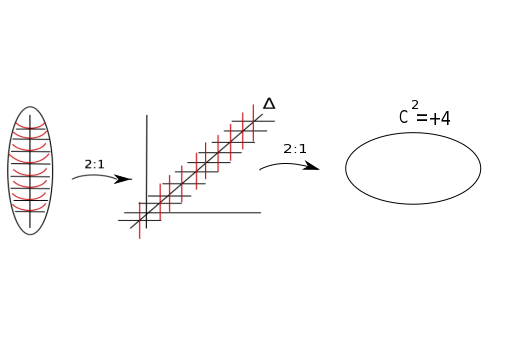}
    \vspace{-1.0cm}
  \caption{Coverings}
  \label{fig:covering}
\end{figure}

We can give local coordinates $(x_0,y_0)$, $(x_{\infty}, y_{\infty})$ of $(\PP, C)$ such that
\begin{equation}\label{eq:CoordCoveringConic}
\left\{\begin{matrix}
2 x_0&=&u_0+v_0\\
y_0&=&(u_0 - v_0)^2
\end{matrix} \right. ,
\hspace{0.5cm} \left\{\begin{matrix}
2x_\infty&=&u_{\infty}+v_{\infty}\\
y_\infty&=&(u_{\infty} - v_{\infty})^2
\end{matrix} \right. ,
\hspace{0.5cm} (u_\infty,v_\infty) = \left( \frac{1}{u_0}, \frac{1}{v_0}  \right)
\end{equation}
where $\Delta=\{v_0=u_0\}=\{v_\infty=u_\infty\}$, $C=\{y_0=0\}=\{y_{\infty}=0\}$. So the cocycle is explicitely given by 
\begin{equation}\label{eq:CocycleConic}
\Phi(x_0, y_0)=\left( \frac{1}{x_0}+\frac{y_0}{4x_0^3}+ \frac{y_0^2}{16x_0^5} +\ldots, \frac{y_0}{x_0^4} +\frac{y_0^2}{2x_0^6} +\frac{3y_0^3}{16x_0^8}+\ldots \right).
\end{equation}
As before, we consider diffeomorphisms $$\Phi^i(x,y) =\left( x+\sum_{n\geq 1}a_n^i (x) y^n,  \sum_{n\geq 1}b_n^i (x) y^n\right),$$ with $b_1^i(0)\neq 0$, $i=0, \infty$, obtaining
$$
\Phi^{\infty}\circ \Phi \circ \Phi^0 (x,y) = \left( \frac{1}{x} + \left( \frac{a_1^0(x) }{x^2} + \frac{b_1^0(x)}{4x^3} + \frac{a_1^{\infty}(\frac{1}{x})b_1^0(x)}{x^4}\right)y +\ldots ,  b_1^{\infty}\left(\frac{1}{x}\right)b_1^0(x)\frac{y}{x^4}  +\ldots \right).
$$
As a consequence we see that there is no transverse fibration in the formal neighborhood at first order.
\end{remark}

\subsection{Two fibrations in general position} 
The goal of this section is to show that there are many $(+1)$-neighborhoods with exactly two fibrations;
 surprisingly, they are easy to classify.

Suppose that the germ of $(+1)$-neighborhood $(S,C_0)$ admits two transverse fibrations $\HH_1$ and $\HH_2$,
such that their tangent locus $T = \tang(\HH_1, \HH_2)$ is neither a leaf, nor $C_0$. 
Remember (see Proposition \ref{prop:tang+1})
that $T$ is smooth and intersects $C_0$ transversely in one point $p_0$, say $x=0$.
We say that $\HH_1$ and $\HH_2$ are in {\it general position} near $C_0$ if, moreover,
the curve $T$ is transversal to $\HH_1$ (and therefore $\HH_2$). Note that, maybe translating 
to the neighborhood of a nearby fiber $C_\epsilon$, we can always assume $\HH_1$ and $\HH_2$ in general position.
To state our result, denote 
$$\mathrm{Diff}^{\ge k}(\mathbb C,0):=\{\varphi(z)\in\mathbb C\{z\}\ ;\ \varphi(z)=z+o(z^k)\}$$
the group of germs of diffeomorphisms tangent to the identity at the order $\ge k$ and denote 
$\mathrm{Diff}^{k}(\mathbb C,0):=\mathrm{Diff}^{\ge k}(\mathbb C,0)\setminus\mathrm{Diff}^{\ge k+1}(\mathbb C,0)$
the set of those tangent precisely at order $k$. The group
$$A:=\{\varphi(z)=az/(1+bz)\ :\ a\in\mathbb C^*,\ b\in\mathbb C\}=\mathrm{PGL}_2(\mathbb C)\cap\mathrm{Diff}(\mathbb C,0)$$
acts by conjugacy on each $\mathrm{Diff}^{\ge k}(\mathbb C,0)$ and therefore on $\mathrm{Diff}^{k}(\mathbb C,0)$.

\begin{thm}\label{thm:TangGenrericModuli}
Germs of $(+1)$-neighborhoods $(S,C_0)$ that admit two transversal fibrations $\HH_1$ and $\HH_2$
in general position are in one to one correspondance with the quotient set 
$$\mathrm{Diff}^1(\mathbb C,0)/A.$$
\end{thm}

\begin{proof}
Like in the proof of Theorem \ref{thm:RigiditySpecialExample}, consider $h_1,h_2 : S \rightarrow \Pp^1$ 
the first integrals of $\HH_1$ and $\HH_2$ whose restrictions on $C_0$ coincide with the global parametrization $x:C_0\stackrel{\sim}{\to}\mathbb P^1$. Now, consider the map
$$
\pi :S \rightarrow  \Pp^1 \times \Pp^1\ ;\  \pi(p) = (h_1(p), h_2(p))
$$
and denote by $\Sigma:=\pi(T)$ the critical locus.  One can check that, in the neighborhood of $p_0$, $\pi$ is a $(k+1)$-fold cover ramifying over $\Sigma$;
indeed, the two fibers $\{h_1=0\}$ and $\{h_2=0\}$ have a contact of order $k+1$ at $p_0$ (see \cite[Section 5]{MaycolFrank} for details). In fact, we can extend
this ramified cover over the neighborhood of $\Delta\subset\mathbb P^1\times\mathbb P^1$ and get 
$$\pi:(S,C_0\cup C_1\cup\cdots\cup C_k)\stackrel{(k+1):1}{\longrightarrow}(\mathbb P^1\times\mathbb P^1,\Delta)$$
where $C_0,\ldots,C_k$ are the preimages of $\Delta$, that intersect transversely at $p_0$. After selecting 
one of them, say $C_0$, we get our initial map $\pi:(S,C_0)\to(\mathbb P^1\times\mathbb P^1,\Delta)$.
From this point of view, it becomes clear that the lack of unicity comes from the choice $C_0$ among 
$C_0,\ldots,C_k$, and the corresponding neighborhoods might be not isomorphic: $(S,C_i)\not\simeq(S,C_0)$.

However, in the case $\HH_1$ and $\HH_2$ are in general position, then $k=1$ and the covering,
being of degree $2$, is automatically galoisian: $(S,C_0)\simeq(S,C_1)$. Then we have a one-to-one 
correspondance between smooth curves $\Sigma$ at $q_0$ having a simple tangency with $\Delta$
and $(+1)$-neighborhoods $(S,C_0)$ with $\HH_1$, $\HH_2$ having a simple tangency at $p_0=\{x=0\}$, 
up to isomorphism preserving the fixed parametrization $x:C_0\to\mathbb P^1$. 

Finally, the change of parametrization induces a diagonal action of the group $A$ on $\mathbb P^1\times\mathbb P^1$.
Viewing $\Sigma$ as the graph of a germ of diffeomorphism $v=\varphi(u)$, we see that $\varphi(u)=u+cu^2+\cdots$, 
$c\not=0$, and the diagonal action induced by $A$ is an action by conjugacy on $\varphi$, whence the result.
\end{proof}

\begin{remark}If we consider the very special $(+1)$-neighborhood $(S,C_0)$ studied along section \ref{tancency-along-C_0},
after specializing to the neighborhood of any deformation $C_\varepsilon\not=C_0$, 
we get a new $(+1)$-neighborhood $(S,C_\varepsilon)$ with two fibrations in general position.
This neighborhood does not depend on the choice of $\varepsilon$ since $\mathrm{SL}_2(\mathbb C)$
acts transitively on those rational curves; it is easy to check that it corresponds to the class of the diffeomorphism
$\varphi(u)=u/(1-u)$, i.e. to a curve $\Sigma$ given by a bidegree $(2,2)$ curve (a deformation of $\Delta$).
\end{remark}

\subsection{Proof of Theorem \ref{3-fibrations}}\label{sec:proof3fibrat}
Before proving the theorem we make some considerations. First of all, recall (see Proposition \ref{prop:NormalFormFibration})
that there are normal forms 
$$
\Phi = \left(\frac{1}{x}, \frac{y}{x} + \sum_{V(-2,3)}b_{m,n}x^my^n\right),  
\Phi_i = \left(\frac{1}{x}, \frac{y}{x} + \sum_{V(-2,3)}b^i_{m,n}x^my^n \right),\hspace{0.2cm} i=1,2,
$$
compatible with $\HH$, $\HH_1$ and $\HH_2$ respectively, and denote by $\vartheta_i=(\alpha_i, \beta_i, \gamma_i,1)$ 
the parameter corresponding to the change of cocycle  from $\Phi$ to $\Phi^i$, $i=1,2$.

\xymatrix{
&&&&& \Phi  \ar[ld]_{(\alpha_1, \beta_1, \gamma_1, 1 )} \ar[rd]^{(\alpha_2, \beta_2, \gamma_2, 1 )} & \\
&&&&\Phi_1  & & \Phi_2 
}
All along the section, we assume moreover $(\alpha_1, \beta_1, \gamma_1)\not=(\alpha_2, \beta_2, \gamma_2)$ 
and are both $\not=(0,0,0)$ (otherwise two of the three fribations coincide).

\begin{lemma}\label{(a1,0,0,1), (0,b2,0,1)}
If $b_{-2,3}=0$, $\vartheta_1=(\alpha,0,0)$ and $\vartheta_2=(0,\beta,0)$,  $\alpha\beta\not=0$,
then $b_{m,n}=0$ for every $(m,n)$ and the neighborhood is linearizable.
\end{lemma}

\begin{proof}
From formula (\ref{eq:ActionCoeffCocycle}), we already get  $0=a^{1}_{-3,4}= \alpha b_{-3,4}$ and $0=a^{2}_{-3,4}= \beta b_{-2,4}$. Here, the notation $a^i_{n,m}$ stands for the $a$-coefficients of $\Phi^i$ obtained by applying 
the change $\vartheta_i$ to a general cocycle $\Phi$.
Assume by induction that $b_{m,k}=0$ for every $k<n$. Then 
$$
\Phi = \left( \frac{1}{x}, \frac{y}{x} + (\frac{b_{-2,n}}{x^2} + \ldots + \frac{b_{1-n,n}}{x^{n-1}})y^n + \ldots \right).
$$
The change of coordinates sending $\Phi$ to $\Phi_{1}$ takes the form (see Proposition \ref{PropMishustinFreedom})
$$
\Psi_1^{0} = \left(\frac{x}{1+\alpha y} -\alpha b_{-2,n}y^n+o(y^n), \frac{y}{1 + \alpha y} \right), \hspace{0.3cm} 
\Psi_1^{\infty} = (x+\alpha y, y),
$$
and we can check by direct computation that  
$$
\Phi_1 = \left(\frac{1}{x} + \alpha (\frac{b_{-3,n}}{x^2} + \ldots + \frac{b_{1-n,n}}{x^{n-1}})y^n+o(y^n), \frac{y}{x}+o(y) \right)
$$
so that we deduce $b_{m,n}=0$ for $m\not=-2$. On the other hand
$$
\Psi_2^{0} = (x+\beta y, y), \hspace{0.3cm}  \Psi_2^{\infty} =\left(\frac{x}{1+\beta y} +\beta b_{1-n,n}y^n+o(y^n), \frac{y}{1 +\beta y} \right),
$$
we see in a similar way that $a^2_{-3,n}=\beta b_{-2,n}=0$. We conclude by induction.
\end{proof}

\begin{lemma}\label{(a1,0,0,1), (a2,0,0,1)}
If $b_{-2,3}=0$ and $\vartheta_i=(\alpha_i,0,0)$ with $\alpha_1\not=\alpha_2$ both non zero, 
then $b_{m,n}=0$ for every $(m,n)$ and the neighborhood is linearizable.
\end{lemma}
\begin{proof}
We prove by induction on $n$, in a very similar way, that $b_{m,n}=0$ for $m\not=-2$, until $n=7$:
for instance, at each step, we get $a^1_{m,n}=\alpha_1b_{m,n}$ (we use only two fibrations so far).
Then, for $n=8$ we find that $a^{i}_{-3,8}=\alpha_i (b_{-3,8} - \alpha_i b_{-2,4})=0$.
Of course, since $\alpha_1\not=\alpha_2$ we get both $b_{-3,8}=b_{-2,4}=0$.
For $n>8$, the induction shows that $b_{-2,n-4}=b_{-3,n}=\cdots=b_{1-n,n}=0$ and we get the result.
\end{proof}

\begin{proof}[Proof of Theorem \ref{3-fibrations}]
We consider the following cases.

\textbf{Case 1: $\tang(\HH, \HH_1) \cap   \tang(\HH, \HH_2)=\emptyset$.}  We change parametrization $x:C_0\to\mathbb P^1$ so that   $\tang(\HH, \HH_1) \cap C_0 =\{x=0\}$ and $\tang(\HH, \HH_2) \cap C_0 =\{x=\infty \}$, which implies $\beta_1 = \alpha_2 =0$. Note that $\alpha_1 \beta_2 \neq 0$ because we are not in the case of tangency along $C_0$.

If $b_{-2,3}\not=0$, we can assume after change of coordinate $\vartheta=(0,0,0,\theta)$ that $b_{-2,3}=1$. 
We consider the equations $a^1_{m,n}=0$ and $a^2_{3,n}=0$ for $n\le7$, we can solve and express 
all $b_{m,n}$, $n\le7$, in terms of $\alpha_1$, $\gamma_1$, $\beta_2$ and $\gamma_2$.
We replace them in the remaining coefficients $a^2_{m,n}$  for $n\le7$, $m\not=3$, 
and use Groebner basis in order to rewrite the ideal 
$$
\langle a^2_{-4,5}, a^2_{-4,6},  a^2_{-5,6}, a^2_{-4,7}, a^2_{-5,7}, a^2_{-6,7}  \rangle = \langle \gamma_2, \alpha_1 \beta_2 \rangle
$$
but this implies that $\alpha_1 \beta_2 =0$, which is not possible.

If $b_{-2,3}=0$ with a similar argument we arrive in $\gamma_1 = \gamma_2 =0$ and thus we conclude by lemma \ref{(a1,0,0,1), (0,b2,0,1)}.\\

\textbf{Case 2: $\tang(\HH, \HH_1) \cap   \tang(\HH, \HH_2)\cap C_0=\{x=0\}$.} 
In this case we can assume $\vartheta_i=(\alpha_i,0,\gamma_i,1)$ with $\alpha_1 \alpha_2 \neq 0$. 
Suppose first $(\alpha_1,\gamma_1)$ not parallel to $(\alpha_2, \gamma_2)$. 

If $b_{-2,3} = 1$, we use equations $a^1_{m,n}=0$, $3 \leq m < n \leq 7$, and  $a^2_{-3,5}= a^2_{-3,6}=a^2_{-3,7}=0$ in order to find all $b_{m,n}$ with $n\le7$ except $b_{-2,7}$ in terms of $\alpha_1$, $\gamma_1$, $\alpha_2$ and $\gamma_2$. Replacing them in $a^2_{-6,7}$ we obtain $\gamma_1=0$ or $\gamma_1 = 2$. The former case implies, by using $a^2_{-3,4}$, that $\gamma_2=2$ and this give us the equation $a^2_{-4,6}=6 \alpha_1 = 0$, impossible. On the other hand, the last case gives us $a^2_{-5,7} = 12(\alpha_1 \gamma_2 - \alpha_2 \gamma_1) =0$, and this also contradicts our hypothesis. We conclude that we are never in this case.

If $b_{-2,3}=0$, we use equations $a^1_{m,n}=0$, $3 \leq i < j \leq 7$, and  $a^2_{-3,5}= a^2_{-3,6}=a^2_{-3,7}=0$ in order to find 
all $b_{m,n}$ with $n\le7$ except $b_{-2,7}$ in terms of $\alpha_1$, $\gamma_1$, $\alpha_2$ and $\gamma_2$.
Replacing them in $a^2_{-3,4}$, $a^2_{-4,5}$ and $a^2_{-4,6}$, we arrive in $\gamma_1 = \gamma_2 =0$. We are in the hypothesis of lemma \ref{(a1,0,0,1), (a2,0,0,1)} and therefore we conclude the theorem.

Finally we consider the case $(\alpha_2, \gamma_2) = \lambda (\alpha_1, \gamma_1)$, for $\lambda \neq 0,1$. 

If $b_{-2,3}=0$, from $a^1_{-3,4} = a^2_{-3,4} = 0$ we obtain $\gamma_1 = \gamma_2 = 0$, 
and then we are able to apply lemma \ref{(a1,0,0,1), (a2,0,0,1)}. 

If $b_{-2,3}=1$, again from $a^1_{-3,4} = a^2_{-3,4} = 0$, we obtain $\gamma_1 = \gamma_2 = 0$ and $a^1_{-3,4} = \alpha_1 b_{-3,4}=0$. Now, from $a^1_{-3,5} = a^1_{-4,5} = 0$ we also get $b_{-3,5} = b_{-4,5} = 0$.  Therefore $a^1_{-3,6}= \alpha_1(b_{-3,6} -\alpha_1)=0$ and $a^2_{-3,6}= \alpha_2(b_{-3,6} -\alpha_2)=0$, which is a contradiction since $\alpha_1 \neq \alpha_2$.	 \end{proof}

\section{Application to Painlev\'e equations}\label{sec:Painleve}

The Painlev\'e equations provide projective structures whose geodesics are the graphs of Painlev\'e transcendents.
We provide the list of coefficients in table \ref{table:Painleve}. They depend on coefficients $\alpha,\beta,\gamma,\delta\in\CC$.

{\tiny
\begin{table}[h]
\caption{Painlev\'e equations}
\begin{center}
\begin{tabular}{|c|c|c|c|c|}
\hline
 equation & A & B & C & D \\
 \hline
PI & $6y^2+x$ & $0$ & $0$ & $0$ \\
\hline
PII & $2y^3+xy+\alpha$ & $0$ & $0$ & $0$ \\
\hline
PIII & $\frac{\alpha y^2+\beta}{x}+\gamma y^3+\frac{\delta}{y}$ & $-\frac{1}{x}$ & $\frac{1}{y}$ & $0$ \\
\hline
PIV & $\frac{3}{2}y^3+4xy^2+2(x^2-\alpha)y+\frac{\beta}{y}$ & $0$ & $\frac{1}{2y}$ & $0$ \\
\hline
PV &   
$\frac{(y-1)^2}{x^2}\left(\alpha y+\frac{\beta}{y}\right)
+\frac{\gamma y}{x}+\frac{\delta y(y+1)}{y-1}$ & $-\frac{1}{x}$ & $\frac{1}{2y}+\frac{1}{y-1}$ & $0$ \\
\hline
PVI & 2$\frac{y(y-1)(y-x)}{x^2(x-1)^2}\left(\alpha+ \frac{\beta x}{y^2}+ \frac{\gamma (x-1)}{(y-1)^2}+ \frac{\delta x(x-1)}{(y-x)^2}\right)$ 
& $-\left(\frac{1}{x}+\frac{1}{x-1}+\frac{1}{y-x}\right)$ & $\frac{1}{2}\left(\frac{1}{y}+\frac{1}{y-1}+\frac{1}{y-x}\right)$ & $0$ \\
\hline
\end{tabular}
\end{center}
\label{table:Painleve}
\end{table}
}

Painlev\'e equations have been found by Painlev\'e more than one hundred years ago when he was looking for 
new second-order differential equations of the form $y''=f(x,y,y')$, with $f$ polynomial, satisfying the Painlev\'e Property:
there is a finite set $S=\{x_1,\ldots,x_n\}\in\CC$ such that any local solution can be meromorphically continued 
along any path avoiding the set $S$. Linear differential equations have this property. Poincar\'e and Fuchs proved
that first-order differential equations $f(x,y,y')=0$ having that property either can be integrated by algebraic functions,
or can be reduced by change of coordinates to Riccati or Weierstrass differential equations.
Painlev\'e proved that a second-order differential equations $y''=f(x,y,y')$ with the Painlev\'e Property can be reduced
to an equation of Table \ref{table:Painleve} if it is {\it irreducible}, i.e. if it cannot be reduced to previously known cases 
(linear, first-order, quadrature).
In fact, due to some mistakes, he forgot some cases and the list was completed by R. Fuchs and Gambier. 
The Painlev\'e Property of these equations was proved by Okamoto by reducing the singularities of 
the Painlev\'e foliation by blow-ups. Painlev\'e equations can also be interpreted as differential equations 
for coefficients of isomonodromic deformations of linear differential systems. This has been proved by R. Fuchs 
in the Painlev\'e VI case, and 
later by Jimbo-Miwa-Ueno for all other cases, which turn out to be more complicated, because involving
systems with irregular singular points with a study of Stokes matrices. It turns out that this isomonodromy
property implies the Painlev\'e Property as shown by Malgrange.
The proof of irreducibility is more recent. It is a transcendance result, that has been proved by Nishioka and Umemura 
in terms of non linear differential Galois theory (see also \cite{CasaleIrred,SergeFrank}). The meaning of irreducibility
is that Painlev\'e equations are as general as a second-order differential equation can be, and more complicated that 
first order. For instance, Painlev\'e foliations cannot be contained in a codimension one foliation 
(see for instance \cite{CasaleIrred,SergeFrank}). 
This implies in particular that the corresponding projective structure on $\CC^2\ni(x,y)$ cannot be globally foliated,
defined by a Frobenius integrable Riccati equation on $\Pp(T\CC^2)$ (see section \ref{Sec:FlatFibration}).
We expect similar transcendental result even in restriction to smaller open sets. 
In that direction, we prove in many cases (see Propositions \ref{Prop:PIPII} and \ref{Prop:PVI})
that projective structures defined by Painlev\'e equations are not locally foliated, at the neighborhood 
of points $(x,y)=(x_0,y_0)$. 
 
But let us recall first the case of linearisability.
One can check that only the sixth one PVI with parameters $(\alpha,\beta,\gamma,\delta)=(0,0,0,\frac{1}{2})$ is locally linearizable,
and it corresponds to the Picard parameter for Painlev\'e equation (see \cite{GuyPicard}). 
In fact, if we apply the criterium of Liouville (see Proposition \ref{Equations-L1-L2}), then we see that
$L_2=0$ for all Painlev\'e equations, and $L_1=0$ only for the Picard one.

\begin{prop}\label{Prop:PIPII}The first two Painlev\'e equations are not locally foliated.
\end{prop}

\begin{proof} We use Proposition \ref{Prop:FoliatedCriterium} and and prove that we cannot find functions $f,g$ 
satisfying equations (\ref{eq:IntegEqfg}) with the coefficients given by the first two rows of table \ref{table:Painleve}.
Because $C=D=0$, the last equation $g_y=-g^2$ gives 
$$g(x,y)=\frac{1}{y+\psi(x)}$$
for a function $\psi$. Then, using moreover $B=0$, the second equation $f_y=g_x-2fg$ gives
$$f(x,y)=\frac{\varphi(x)-y\psi'(x)}{(y+\psi(x))^2}$$
for a function $\varphi$. Finally, the first equation yields
$$f_x-f^2-Ag+A_y=\frac{h(x,y)}{(y+\psi(x))^4}$$
where $h$ is a polynomial in $y$ with coefficients in $x$ (depending on $\varphi$, $\psi$ and their derivatives).
But the dominant term of $h$ is $6y^5$ and $4y^6$ 
for Painlev\'e equations I and II respectively.
Therefore,  whatever are $\varphi$ and $\psi$, $h$ cannot vanish identically.
\end{proof}

Similar computations seem out of reach for all cases, especially Painlev\'e VI case. Instead of this,
we will use our results proved along the paper to provide an alternate and more geometric approach
to prove non existence of foliated structure in the case of Painlev\'e VI. In fact, we will use the monodromy
representation that we now explain. 

The projective structure defined by Painlev\'e VI equation is holomorphic on 
$$(x,y)\in\Omega=\CC^2\setminus\{xy(x-1)(y-1)(y-x)=0\}.$$
Consider the geodesic (or Painlev\'e) foliation in the open chart
$\Omega\times\CC\ni(x,y,z)$ by setting $z=\frac{dy}{dx}$. One easily check that $\Sigma:=\{x=x_0\}\subset\Omega$ 
defines a cross-section for the foliation, for any $x_0\not=0,1$. The Painlev\'e Property of the equation 
allows us to continue meromorphically each local solution along any loop $\gamma:[0,1]\to\Omega$, $\gamma(0)=\gamma(1)=x_0$,
and therefore define a monodromy map $\phi_\gamma:\Sigma\dashrightarrow\Sigma$. This return map is only defined 
outside of an analytic set, but can be inversed. One can built a larger fiber bundle $x:\hat\Omega\to(\CC\setminus\{0,1\})$ 
with respect to which the foliation has a well-defined holonomy, therefore providing a representation 
$$\pi_1(\CC\setminus\{0,1\},x_0)\to\mathrm{Bihol}(\hat\Sigma)$$
into the group of biholomorphisms of a larger transversal $\hat\Sigma:=\{x=x_0\}\subset\hat\Omega$
(see \cite{Okamoto,DM,IIS,SergeFrank}). Precisely, Okamoto constructed in \cite{Okamoto} a partial compactification $\hat\Omega$
with the lifting path property to provide a geometric counterpart to Painlev\'e Property. The transversal $\hat\Sigma$
in this larger space is usually called Okamoto's space of initial conditions. The monodromy representation has been 
defined by Dubrovin-Mazzocco in \cite{DM} for special parameters, and was generalized to all parameters 
by Iwasaki (see \cite{IIS,SergeFrank} for a global picture). It turns out that the monodromy group corresponds
via the Riemann-Hilbert correspondance with an action of the Mapping-Class-Group of the $4$-punctured sphere
on the character variety for $\mathrm{SL}(2,\CC)$.
Some dynamical properties have been studied by Dubrovin-Mazzocco, Iwasaki-Uehara and Cantat using 
this later point of view. We will use a result obtained in \cite{SergeFrank}.

\begin{prop}\label{Prop:PVI}The Painlev\'e VI equation is not locally foliated, 
except for the Picard parameter $(\alpha,\beta,\gamma,\delta)=(0,0,0,\frac{1}{2})$.
\end{prop} 

\begin{proof} 
If it is foliated at some point $p\in\Omega$, then Corollary \ref{lem:ProlongationFol} allows us to extend meromorphically the foliated
structure along any path inside $\Omega$. Assume first that it is uniform, 
i.e. that we have a global meromorphic foliated structure on $\Omega$.
Then we can apply results of \cite{SergeFrank} about the monodromy of the Painlev\'e VI equation. 
It is shown in \cite[Theorem D]{SergeFrank} that the monodromy cannot preserve a foliation, and we get a contradiction.
Therefore, the foliated structure cannot be uniform. In fact, if it has finitely many determinations, then it defines a web
which is invariant under the monodromy, and again \cite[Theorem D]{SergeFrank} gives a contradiction.
Finally, if it has at least 3 determinations, then this implies by our 
Theorem \ref{3-fibrations} that the projective structure is linearizable, and therefore that we are in the  case of Picard parameters as in the statement.
\end{proof}

\begin{example}
In \cite{Hurtubise}, Hurtubise and Kamran provide linking formulae between 
Cartan invariants for projective structures, and coefficients $a_{m,n}$ and $b_{m,n}$
of normalized cocycle in Theorem \ref{Mishustin} up to order $n=7$. 
In a similar way, one can easily compute the normalized cocycle of Painlev\'e I equation up to order $n=11$
with Maple:
$$\Phi_{P_I}=(a(x,y),b(x,y))\ \ \ \text{with}$$
$$\left\{\begin{matrix}
a(x,y)=\frac{1}{x}+63\frac{y^6}{x^5}-\frac{5311}{2}\frac{y^9}{x^7}+\frac{209039}{154}\frac{y^{11}}{x^{10}}+o(y^{11})\hfill\\
b(x,y)=\frac{y}{x}-\frac{11}{2}\frac{y^4}{x^3}+\frac{489}{7}\frac{y^7}{x^5}+\frac{1243}{840}\frac{y^9}{x^8}-\frac{8137}{7}\frac{y^{10}}{x^{7}}+o(y^{11})
\end{matrix}\right.$$
This corresponds to the cocycle of $U^\vee$ dual to the projective structure induced by $y''=6y^2+x$
at $(\CC^2,0)$. However, it is difficult to compute many other examples as complexity increases fastly
with the size of entries.
\end{example}

\section{Concluding remarks}

\subsection{Coordinates on $C$}
In the classification, we have fixed a coordinate $x:C\to\mathbb P^1$. One could consider the action
of Moebius transformations on $C$ and therefore on $x$. For instance, the action of homotheties $x\mapsto\lambda x$
on normal forms is easy:
$$a_{m,n}'=\lambda^{m-1}a_{m,n}\ \ \ \text{and}\ \ \ b_{m,n}'=\lambda^{m-1}b_{m,n}.$$
If we add this action to the $4$-parameter group $\Gamma$, then orbits correspond to the analytic class
of $(S,C\supset\{p_0,p_\infty\})$ where we have fixed two points $p_i=\{x=i\}$ without fixing the coordinate on $C$.
If we blow-up $p_0$ and $p_\infty$, and then contract the strict transform of $C$, then we get a germ
$(\tilde S,C_0\cup C_\infty)$ of neighborhood of the union of two rational curves $C_0$ and $C_\infty$
(exceptional divisors) with self-intersection $C_i\cdot C_i=0$, that intersect transversally at a single point $p=C_0\cap C_\infty$
(the contraction of $C$). In fact, we can reverse this construction and have a one-to-one corespondance 
$$(S,C\supset\{p_0,p_\infty\})\ \leftrightarrow\ (\tilde S,C_0\cup C_\infty)$$
so that analytic classifications are the same. We note that the action of other Moebius transformations 
on normal forms $\Phi$ are much more difficult to compute. 

\subsection{The case of general positive self-intersection $C\cdot C>1$}\label{sec:GeralSelfInt}
Mishustin gave in \cite{Mishustin} a normal form for $(d)$-neighborhoods
for arbitrary $d\in\mathbb Z_{>0}$ and the story is similar to the case $d=1$. 
More geometrically, we can link the general case to the case $d=1$ as follows.
Let $d>1$ and $(S,C)$ a $(d)$-neighborhood. Then, maybe shrinking $S$, 
the topology of $(S,C)$ is the same than the topology of $(N_C,0)$, in particular:
$$\pi_1(S\setminus C)\simeq\mathbb Z/<d>$$
i.e. the fundamental group of the complement of $C$ is cyclic of order $d$. We can 
consider the corresponding ramified cover 
$$(\tilde S,\tilde C)\stackrel{k:1}{\rightarrow} (S,C)$$
totally ramifying at order $d$ over $C$, and inducing the cyclic cover of order $d$
over the complement $S\setminus C$. If we do this with $S$ being the total space 
of $\mathcal O_{\mathbb P^1}(k)$, then $\tilde S$ will be the total space of $\mathcal O_{\mathbb P^1}(1)$,
the neighborhood of a line in $\mathbb P^2$. Likely as in the linear case, the lifted curve $\tilde C$
will have self-intersection $\tilde C\cdot \tilde C=1$ and we are back to the case $d=1$. 
Moreover, $\tilde S$ is equipped with the Galois transformation, of order $d$, which has $\tilde C$ as a fixed point curve. 

\begin{prop}\label{prop:PositiveSelfIntersection}
Isomorphism classes $(S,C)$ of germs of $(d)$-neighborhoods of the (para\-metrized) rational curve $x:C\to\mathbb P^1$
are in one-to-one correspondance with isomorphism classes of germs of $(+1)$-neighborhoods $(\tilde S,C)$
equipped with a cyclic automorphism of order $d$ fixing $C$ point-wise.
\end{prop}

We can also see a $(d)$-neighborhood $(S,C)$ from the following point of view.

\begin{prop}
Given $p_1, \ldots, p_{d+1}$ distinct fixed points on $C$. Then, there exists a one-to-one correspondance between $(d)$-neighborhoods $(S,C)$ and neighborhood germs $(V,D)$ of a bouquet $D=D_1 \cup \ldots \cup D_{d+1}$ of rational curves of zero self-intersection intersecting transversely at  one point.  
\end{prop}

\begin{figure}[ht!]
\includegraphics[scale=0.4]{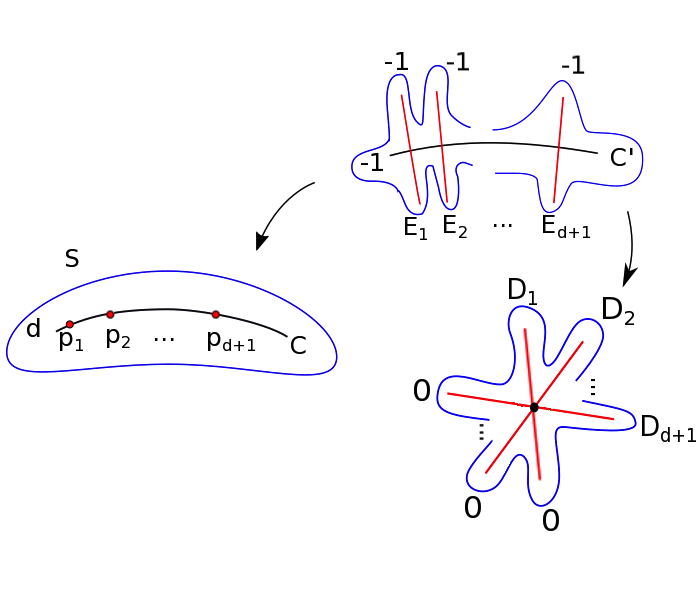}
\caption{0-Neighborhood}
\end{figure}

\begin{proof}
 After blowing-up $(d+1)$ points on $C$, we get a neighborhood of $C' \cup E_1 \ldots \cup E_{d+1}$ where $C'$ is the strict transform of $C$ and $E_i$'s are the exceptional divisors. Moreover, all these $d + 2$ rational curves have self-intersection $-1$. After contracting $C'$, we get a neighborhood $V$ of $d + 1$ rational curves $D_1 \cup \ldots \cup D_{d+1}$ intersecting at $1$ point and having self-intersection $0$. We obtain like this a neighborhood of $D=D_1 \cup \ldots \cup D_{d+1}$ and this process can be reversed.
\end{proof}

\begin{remark}
The germ of surface $(V,D)$ is obtained by gluing $(0)$-neighborhoods $(V_i, D_i)$, which are trivial by \cite{Savelev}, by elements of $\mathrm{Diff}(\CC^2,0)$.
\end{remark}

\subsection{Open questions}

\begin{itemize}
\item {\bf Riemannian metrics}: projective structures arising from holomorphic metrics are rare,
there are infinitely many obstructions (see \cite{BDE}).
Can we characterize these Riemannian projectives structures, or their dual surfaces $U^\vee$, in a geometric way ? 

\item The linear projective structure is Riemannian for several metrics (arbitrary constant curvature).
For a generic Riemannian projective structure, is the Riemannian metric determined up to a scalar ?

\item {\bf Algebraic dimension}: generic $(+1)$-neighborhoods $U^\vee$ do not carry non constant meromorphic functions (see \cite{MaycolFrank2}). Does there exist examples with  
with function field having transcendance dimension one over $\mathbb C$ ?

\item {\bf Painlev\'e equations}: what can be said about those $U^\vee$ associated to a Painlev\'e equation ?
We know that they are rarely Riemannian (see \cite{ContattoDunajski}).
Do they provide explicit examples without non constant meromorphic functions~?
Our Propositions \ref{Prop:PIPII} and \ref{Prop:PVI} are first step towards this direction.

\item {\bf Singularities}: global projective structures on compact surfaces have been classified
in \cite{KO}, they impose strong constraints on the surface. It would be interesting to investigate
a singular version of projective structure, for instance defined by a logarithmic affine connection
(see \cite{Briancon,NY}), in view of examples of orbifold uniformization in dimension 2. 
In this direction, logarithmic singularities of affine structures are studied in \cite{LL}.
\end{itemize}

\end{document}